\documentclass[reqno,11pt]{amsart}
\usepackage{amsfonts}
\usepackage{graphicx}
\usepackage{amsfonts,amsmath, amssymb}
\usepackage{hyperref}
\usepackage{marginnote}
\usepackage{color}
\usepackage{cite}
\usepackage{bm}
\usepackage{enumerate}
\usepackage{cancel}
\numberwithin{equation}{section}

\numberwithin{equation}{section}
\theoremstyle{plain}
\allowdisplaybreaks
\newcommand{\beq}{\begin{equation}}
\newcommand{\eeq}{\end{equation}}
\newcommand{\ben}{\begin{eqnarray}}
\newcommand{\een}{\end{eqnarray}}
\newcommand{\beno}{\begin{eqnarray*}}
\newcommand{\eeno}{\end{eqnarray*}}

\newtheorem{theorem}{Theorem}[section]

\newtheorem{lemma}[theorem]{Lemma}
\newtheorem{proposition}[theorem]{Proposition}

\newtheorem{remark}[theorem]{Remark}

\begin{document}
\date{\today}
\title
{Unstable mode around the 3D boundary layer flow}

\author[C.-J. Liu]{Cheng-Jie Liu}
\address[C.-J. Liu]{School of Mathematical Sciences, LSC-MOE, CMA-Shanghai and Institute of Natural Sciences, Shanghai Jiao Tong University, Shanghai, China}
\email{liuchengjie@sjtu.edu.cn}

\author[M. Ma]{Mengjun Ma}
\address[M. Ma]{School of Mathematics, Sun Yat-Sen University, 510275 Guangzhou, China}
\email{mamj7@mail2.sysu.edu.cn}

\author{Di Wu}
\address[D. Wu]{School of Mathematics, South China University of Technology, Guangzhou, 510640,  P. R. China}
\email{wudi@scut.edu.cn}

\author[Z. Zhang]{Zhu Zhang}
\address[Z. Zhang]{Department of Applied Mathematics, The Hong Kong Polytechnic University, Hong Kong}
\email{zhuama.zhang@polyu.edu.hk}

\begin{abstract}
We study the stability properties of boundary layer-type shear flows for the three-dimensional Navier-Stokes equations in the limit of small viscosity $0<\nu\ll 1$. When the streamwise and spanwise velocity profiles are linearly independent near the boundary, we construct an unstable mode that exhibits rapid growth at the rate of $e^{t/\sqrt{\nu}}$. Our results reveal an analytic instability in the three-dimensional Navier-Stokes equations around generic boundary layer profiles. This instability arises from the interplay between spanwise flow and three-dimensional perturbations, and does not occur in purely two-dimensional flows.
\end{abstract}
\maketitle
\tableofcontents
\section{Introduction}

\subsection{Background}

Understanding the behavior of high-Reynolds-number flows near physical boundaries is a fundamental problem in mathematical fluid mechanics. The theoretical framework for describing this behavior is the Prandtl's boundary layer theory \cite{Prandtl-1904}, which provides an asymptotic approximation for the fluid flow near a rigid wall. However, recent results have shown that certain instability mechanisms can lead to the rapid growth of infinitesimal disturbances, potentially calling into question the validity of Prandtl's ansatz in centain regimes. Therefore, investigating the stability properties of boundary layer profiles is crucial for identifying the conditions under which the Prandtl's theory remains valid.\\

In two-dimensional setting, the boundary layer exhibits three types of instabilities. The first one is the instability inherent to the Prandtl equation. G\'erard-Varet and Dormy \cite{GD} has shown that non-monotone profiles can induce instability in the Gevrey-$2$ class for the 2D Prandtl equation. The second type is induced by the inflection point within boundary layer profiles. In this case, the Rayleigh criterion fails, and the linearized Navier-Stokes equations become strongly ill-posed in non-analytic function spaces. For results on nonlinear instabilities of such profiles, see \cite{BG,G2000,GN}. As a consequence, the stability of generic boundary layers (or the inviscid limit) can only be established within the analytic framework; see, for example, \cite{FM-TT-2018,KVW,NN,YM2014,SC2,WC-WY-ZZ,WC-WY-ZZ1}. The third type of instability is driven by the destabilizing effect of small viscosity: for monotone and concave profiles that are inviscidly stable, the linearized Navier-Stokes system becomes unstable in the Gevrey-$\frac32$ class. This instability is closely related to the formation of Tollmien-Schlichting waves, which was rigorously justified for linearized Navier-Stokes equations by Grenier, Guo, and Nguyen  \cite{GGN-DMJ}, and in the nonlinear setting by Grenier and Nguyen \cite{GN1}. We also refer to \cite{BG2} for results on this type of instability in different physical settings. On the other hand, the validity of Prandtl expansion in the critical Gevrey-$\frac32$ space has been established by G\'erard-Varet-Maekawa-Masmoudi \cite{GMM-duke,GMM-arxiv}. We also refer to \cite{CWZ1} for results on $L^\infty$ stability up to the initial time. In the case of compressible fluids, T-S instabilty has been established for subsonic boundary layers \cite{MWWZ-TS-proc,YZ}, while the Mack mode has been demonstrated in supersonic boundary layers \cite{MWWZ-2024supersonic}. Finally, the aforementioned instabilities are absent in steady or symmetric flows. As a result, the stability of boundary layers can be established in Sobolev spaces; see \cite{GM,GI,IM,MT-2010} for related results. For a more comprehensive review of developments in boundary layer theory, we refer to the monograph by Maekawa and Mazzucato \cite{MZ}.
\\

Although tremendous progress has been achieved in the two-dimensional boundary layer theory, the mathematical analysis of three-dimensional boundary layers remains quite limited. In three dimensions, the (in)stability mechanism of boundary layers are more complicated due to the possible presence of secondary flows, as observed in physical experiments \cite{RS,SY1985,SRW,Sch}.  When the two tangential velocity profiles are linearly independent, Liu, Wang, and Yang \cite{LWY} constructed a growing mode in Gevrey-$2$ space for the 3D Prandtl equation. For the well-posedness theory in the case of dependent profiles, we refer to \cite{LWY1,LWY2}. Recently, Li, Masmoudi, and Yang \cite{LWX2022} established stability for the 3D Prandtl equation in the Gevrey-2 space. Remarkably, the critial Gevrey index for the stability of 3D Prandtl equation coincides with that of the 2D case. However, Gallaire, G\'erard-Varet, and Rousset \cite{GGR} demonstrated that certain special flows (without boundary) that are stable in two dimensions can become unstable for the 3D Navier-Stokes system. These developments motivate a natural question: at high Reynolds numbers, can three-dimensional boundary layer flows exhibit stronger instabilities than their two-dimensional counterparts? In this paper, we reveal an instability mechanism arising from the combined effect of secondary flow and three-dimensional perturbations, which is stronger than the classical Tollmien-Schlichting instabilty in two dimensions.

\subsection{Problem and Main results}

Consider the three-dimensional incompressible Navier-Stokes equations in the half-space $\Omega=\{(x,y,z)\in \mathbb{R}_+^3\mid~(x,y)\in\mathbb{R}^2, z>0\}:$
\begin{align}\label{NS}
    \left\{
    \begin{aligned}
        &\partial_t \textbf{u}^\nu+\textbf{u}^\nu\cdot\nabla \textbf{u}^\nu+\nabla p^\nu-\nu\Delta \textbf{u}^\nu=\textbf{F}^\nu,
        \\ &\nabla\cdot \textbf{u}^\nu=0,\\
         &\textbf{u}^\nu|_{z=0}=0.
    \end{aligned}
    \right.
\end{align}
Here $\textbf{u}^{\nu}=(u^\nu,v^\nu,w^\nu)$ and $p^\nu$ stand for the velocity field and pressure respectively. We are interested in dynamics at small viscosities (large Reynolds number). Hence, we set $0<\nu\ll1.$

In a local regime along the streamwise direction, the laminar boundary layer flow can be approximated by the following shear flow: $$\textbf{u}_s(t,x,y,z)=\left(u_s(\frac{z}{\sqrt{\nu}}),v_s(\frac{z}{\sqrt{\nu}}),0\right),$$ 
where $u_s$ and $v_s$ are boundary layer profiles in the streamwise and spanwise directions respectively. This is a stationary solution to the 3D Navier-Stokes system \eqref{NS}, provided that a suitable external forcing $\textbf{F}^\nu$ is imposed.

To understand the (in)stability properties of this base flow, we consider the following  linearized incompressible Navier–Stokes equations around $\textbf{u}_s$:
\begin{align}\label{eq:NS-linear}
    \left\{
    \begin{aligned}
        &\partial_{t}\textbf{v}^\nu+u_s(\frac{z}{\sqrt{\nu}})\partial_{x}\textbf{v}^\nu+v_s(\frac{z}{\sqrt{\nu}})\partial_{y}\textbf{v}^\nu
        +\left(\partial_{z}u_s(\frac{z}{\sqrt{\nu}}),\partial_{z}v_s(\frac{z}{\sqrt{\nu}}),0\right)w+\nabla q^\nu-\nu \Delta \textbf{v}^\nu=0,\\ 
         &\nabla\cdot \textbf{v}^\nu=0,\\ 
        &\textbf{v}^\nu|_{z=0}=0,
    \end{aligned}
    \right.
\end{align}
A classical approach to finding instabilities is to seek a wave solution to the linearized system \eqref{eq:NS-linear} of the form
\begin{align}
(\textbf{v}^{\nu}, q^\nu)(t,x,y,z)=e^{-\mathrm{i}\alpha c\tau}e^{\mathrm{i}\sigma  X+\mathrm{i}\beta Y}(\tilde{u},\tilde{v},\tilde{w}, \tilde{q})(Z),\label{ans}
\end{align}
where $(\tau,X,Y,Z)$ are rescaled variables defined by
\begin{align*}
	\tau=\frac{t}{\sqrt{\nu}},\quad X=\frac{x}{\sqrt{\nu}},\quad Y=\frac{y}{\sqrt{\nu}},\quad Z=\frac{z}{\sqrt{\nu}},
\end{align*}
and the parameters $\sigma$ and $\beta$ are the wave numbers in $X$ and $Y$ directions respectively. The complex number $c=c_r+\mathrm{i}c_i$ denotes the wave speed, and $\alpha\triangleq(\sigma^2+\beta^2)^\frac{1}{2}$ is the total wave number. Substituting the ansatz \eqref{ans} into \eqref{eq:NS-linear}, we obtain the following ODE system for the amplitude  $(\tilde{u},\tilde{v},\tilde{w}, \tilde{q})(Z)$:
\begin{align}\label{eq:LNS-sig-be}
	\left\{
	\begin{aligned}
		&(\frac{\sigma}{\alpha}u_s+\frac{\beta}{\alpha}v_s-c)\tilde{u}(Z)+\frac{1}{\mathrm{i}\alpha}
\partial_Zu_s\tilde{w}(Z)+\frac{\sigma}{\alpha}\tilde{q}(Z)-\frac{\sqrt{\nu}}{\mathrm{ i}\alpha}(\partial_Z^2-\alpha^2)\tilde{u}(Z)=0,\\
		&(\frac{\sigma}{\alpha}u_s+\frac{\beta}{\alpha}v_s-c)\tilde{v}(Z)
+\frac{1}{\mathrm{i}\alpha}
\partial_Zv_s\tilde{w}(Z)+\frac{\beta}{\alpha}\tilde{q}(Z)-\frac{\sqrt{\nu}}{\mathrm{ i}\alpha}(\partial_Z^2-\alpha^2)\tilde{v}(Z)=0,\\
		&(\frac{\sigma}{\alpha}u_s+\frac{\beta}{\alpha}v_s-c)\tilde{w}(Z)
+\frac{1}{\mathrm{i}\alpha}
\partial_Z\tilde{q}(Z)-\frac{\sqrt{\nu}}{\mathrm{ i}\alpha}(\partial_Z^2-\alpha^2)\tilde{w}(Z)=0,\\
		&\mathrm{i}\sigma\tilde{u}(Z)+\mathrm{i}\beta\tilde{v}(Z)+\partial_Z\tilde{w}(Z)=0,\\
		&\tilde{u}(0)=\tilde{v}(0)=\tilde{w}(0)=\tilde{u}(\infty)=\tilde{v}(\infty)=\tilde{w}(\infty)=0.
	\end{aligned}
	\right.
\end{align}
If \eqref{eq:LNS-sig-be} admits a nontrivial solution for some wave numbers $\sigma, \beta\in \mathbb{R}$, and a wave speed $c\in \mathbb{C}$ with $c_i>0$, then the boundary layer profile $\textbf{u}_s$ is spectrally unstable. \\

We now state the main result of the paper. Let the streamwise boundary layer profile $u_s(Z)\in C^4(\overline{\mathbb{R}_+})$ such that
\begin{align}
	&u_s(0)=0,~~u_s'(0)>0,~~u_s(Z)>0 \text{ for }Z>0,\nonumber\\
	&\sum_{k=0}^4\sup_{Z\geq 0}\left|\partial_Z^k\big(u_s(Z)-1\big)e^{\eta_0Z}\right|<\infty,~~\text{ for some }\eta_0>0.\nonumber
\end{align}
\begin{theorem}\label{main-result}
   Let $0<\nu\ll 1$. There exists a class of spanwise profile $v_s$, together with wave numbers $\sigma,\beta\in \mathbb{R}$, and a wave speed $c=c_r+ic_i$ with $c_i>0$,
   such that the linearized incompressible Navier–Stokes equations  \eqref{eq:NS-linear} admit a nontrivial solution $(u,v,w)$ of the form 
   \begin{align*}
   	(u,v,w)(t,x,y,z)=e^{-\mathrm i\alpha c\nu^{-\frac{1}{2}}t}e^{\mathrm i\nu^{-\frac{1}{2}}  (\sigma x+\beta y)}\left(\tilde{u}(\frac{z}{\sqrt{\nu}}),\tilde{v}(\frac{z}{\sqrt{\nu}}),\tilde{w}(\frac{z}{\sqrt{\nu}})\right),
   \end{align*}
   where the amplitude functions
   $(\tilde{u},\tilde{v},\tilde{w})(Z)\in H^1(\mathbb{R}_+)$ solve the eigenvalue problem \eqref{eq:LNS-sig-be}. The parameters are taken in the regime
   \begin{align}\label{r}
   \sigma, \beta\ll 1,\qquad
   \alpha\triangleq(\sigma^2+\beta^2)^{\frac{1}{2}}\sim c_r,\qquad 0<c_i\sim \alpha^2.
   \end{align}
   In particular, the solution exhibits the following growth estimate:
   \begin{align}
   	|(u,v,w)(t,x,y,z)|\sim e^{\nu^{-\frac{1}{2}}t}.\label{gr}
   \end{align}
\end{theorem}

Here we give several remarks on Theorem \ref{main-result}.

\begin{remark}
Our result holds for a class of exponential profiles:
 	\begin{align}
 	u_s(Z)=1-e^{-k_{u}Z},~v_s(Z)=v_{\infty}(1-e^{-k_vZ}),\nonumber
 \end{align}
where parameters $k_u,k_v>0$, and $v_{\infty}\in\mathbb{R}$ satisfy
$$k\triangleq\frac{k_v}{k_u}>1, ~~\text{and}~~ k|v_{\infty}|\geq \max\left\{1,\sqrt{\frac{2}{k-1}}\right\}.
$$
Note that base flows $u_s(Z)$ and $v_s(Z)$ can be uniformly concave. Therefore, the instability arises from the interplay between \textbf{spanwise flow} and \textbf{three-dimensional disturbances}, and does not occur in two-dimensional configurations.
\end{remark}
\begin{remark} 	
	The structural conditions for $v_s$ will be specified in Proposition \ref{pSC}. The main assumption is
	\begin{align}
		\partial_Z^2u_s(0)\partial_Zv_s(0)\neq\partial_Z^2v_s(0)\partial_Zu_s(0),\nonumber
	\end{align}
	which implies that the streamwise and spanwise velocity profiles are linearly independent, leading to the emergence of secondary flows near the boundary.  On the other hand, if $u_s$ and $v_s$ are linearly dependent,  the instability mechanism for \eqref{eq:NS-linear} coincides with that of the two-dimensional case.  
\end{remark}

\begin{remark}        
   The growing mode is supported in high frequency regime $n=\frac{\alpha}{\sqrt{\nu}}\sim \nu^{-\frac{1}{2}}$ and grows at the rate $e^{Cnt}$. Therefore, the Navier-Stokes system exhibits an  \textbf{analytic instability} around generic three-dimensional boundary layer profiles, which is stronger than the classical Tollmien-Schlichting instability observed in two dimensions \cite{GGN-DMJ}. In particular, the validity of Prandtl expansion in three dimensions cannot be expected without assuming analytic regularity. This stands in sharp contrast to the two-dimensional case, where validity can be established for data in the Gevrey-$\frac32$ class \cite{GMM-duke}.   Moreover, this instability can be detected in both Euler and Navier-Stokes equations (with small viscosity).
 \end{remark}

    \begin{remark} Let $\delta=\|v_s\|_{L^\infty}$ denote the amplitude of the base flow in spanwise direction $v_s$. Instability occurs for any $\delta>0$, regardless of how small it is. This leads to the following bifurcation scenario:
    	\begin{align}
    	&\delta=0 \Rightarrow \text{Stability in the Gevrey-$3/2$ class for concave profiles $u_s$};\nonumber\\
    	&\delta>0\Rightarrow \text{Instability occurs in the analytic space}.\nonumber
    	\end{align}
      \end{remark}   
  

\subsection{Instability mechanism and roadmap of proof} Let us describe the instability mechanism. In contrast to the two-dimensional case, the three-dimensional Navier-Stokes system \eqref{eq:LNS-sig-be} is more challenging to analyze due to the absence of a stream function representation. Inspired by  \cite{CWZ3DCoutte}, we observe that the vertical velocity component $\tilde{w}$ is decoupled from the other variables. Indeed,  by acting the operator $(\partial_Z^2-\alpha^2)$ to the third equation of $\eqref{eq:LNS-sig-be}$, we obtain 
\begin{align*}
	&\left(\frac{\sigma}{\alpha}u_s(Z)+\frac{\beta}{\alpha}v_s(Z)-c\right)(\partial_Z^2-\alpha^2)\tilde{w}+2\partial_Z\left(\frac{\sigma}{\alpha}u_s(Z)+\frac{\beta}{\alpha}v_s(Z)\right)\partial_Z\tilde{w}\\ &\qquad+\partial_Z^2\left(\frac{\sigma}{\alpha}u_s(Z)+\frac{\beta}{\alpha}v_s(Z)\right)\tilde{w}
+\frac{1}{\mathrm{i}\alpha}
(\partial_Z^2-\alpha^2)\partial_Z\tilde{q}-\frac{\sqrt{\nu}}{\mathrm{i}\alpha}(\partial_Z^2-\alpha^2)^2\tilde{w}=0.
\end{align*}
Then taking divergence of the velocity equations in \eqref{eq:LNS-sig-be}, we get
\begin{align*}	-\frac{1}{\mathrm{i}\alpha}
(\partial_Z^2-\alpha^2)\partial_Z\tilde{q}=2\partial_Z\left(\frac{\sigma}{\alpha}u_s(Z)+\frac{\beta}{\alpha}v_s(Z)\right)\partial_Z\tilde{w}+2\partial_Z^2\left(\frac{\sigma}{\alpha}u_s(Z)+\frac{\beta}{\alpha}v_s(Z)\right)\tilde{w}.
\end{align*}
By adding the above two equations and incorporating the boundary values of $\tilde{w}$, we derive
the following Orr-Sommerfeld equation for $\tilde{w}$:
\begin{align}\label{eq:OS-w}
    \left\{
    \begin{aligned}
		&\mathrm{OS}[\tilde{w}]:=-\varepsilon(\partial_Z^2-\alpha^2)^2\tilde{w}+(U-c)(\partial_Z^2-\alpha^2)\tilde{w}-\partial_Z^2U\tilde{w}=0,\\
		&\tilde{w}(0)=\partial_Z\tilde{w}(0)=\tilde{w}(\infty)=\partial_Z\tilde{w}(\infty)=0,
    \end{aligned}
    \right.
\end{align}
where $\varepsilon\triangleq\frac{\sqrt{\nu}}{\mathrm{ i}\alpha}$. In \eqref{eq:OS-w}, 
\begin{align}	U(Z)=\frac{\sigma}{\alpha}u_s(Z)+\frac{\beta}{\alpha}v_s(Z)\label{U}
\end{align}
is a new base flow induced by the three-dimensional perturbation. Hence, a physical interpretation of our result is that the three-dimensional perturbation produces a strong distortion of the base flow, which triggers an instability mechanism. In two-dimensional settings,  this type of instability does not occur, as both $u_s$ and $v_s$ can be uniformly concave.\\

In Section \ref{sec2}, we specify structural conditions on the spanwise profile $v_s$ so that the resulting base flow $U(Z)$ satisfies 
\begin{align}\label{SC}
	\partial_ZU(0)>0,~~\partial_Z^2U(0)>0,~~\text{and }U(Z)>0 \text{ for }Z>0.
\end{align}
Under these conditions, instability arises in the inviscid equation, and persists in the viscous equation with small viscosities. Further details are provided in Sections \ref{sec-Ray} and \ref{sec-OS} respectively.\\

We now briefly outline the proof.

\begin{itemize}
	\item \underline{Step 1. Inviscid unstable mode. } In Section \ref{sec-Ray}, we construct a homogeneous solution $\varphi_{Ray}(Z)$ exhibiting growth for the Rayleigh equation
	$$\mathrm{Ray}[\varphi_{Ray}]\triangleq(U-c)(\partial_Z^2-\alpha^2)\varphi_{Ray}-U''\varphi_{Ray}=0.
	$$
	Although the classical Rayleigh's criterion is violated due to the convexity condition $U''(0)>0$ in \eqref{SC}, constructing an unstable mode is still very challenging. Compared to  \cite{GGN-DMJ}, our analysis requires a more precise pointwise estimate for $\varphi_{Ray}$. In particular, capturing the inviscid instability requires a more accurate second-order asymptotic expansion of $\varphi_{Ray}(0)$ with respect to small parameters $\alpha$ and $c$.  This is the main technical difficulty of this section. After  careful analysis, we obtain the following asymptotic expansion:
	\begin{align}
	~~~~\varphi_{Ray}(0;c)=-\frac{c_r}{U_{\infty}}+\frac{\alpha}{U'(0)}+O(\alpha^2|\log\alpha|)+\mathrm{ i}\left(-\frac{c_i}{U_\infty^2}+\frac{U''(0)\pi}{U'(0)^3}\alpha c_r+o(\alpha^2)\right),\nonumber
	\end{align}
when $\alpha$ and $c=c_r+ic_i$ are in the regime \eqref{r}.  Thanks to the convexity $U''(0)>0$, we find an unstable eigenvalue 
	\begin{align}
	c_{Ray}\approx\frac{\alpha U_{\infty}}{U'(0)}+\mathrm{ i}\frac{\alpha^2\nonumber U''(0)U_{\infty}^4\pi}{U'(0)^4}.
	\end{align}
	 This implies an inviscid instability of boundary layer profiles under three-dimensional perturbations. The detail will be provided in Theorem \ref{iv}.\\

\item \underline{Step 2. Viscous instability. } Note that $\varphi_{Ray}$ is only an approximate solution to the Orr-Sommerfeld equation, and does not satisfy the full boundary conditions in \eqref{eq:OS-w}. Therefore, it is necessary to construct a homogeneous Orr-Sommerfeld solution $\phi_s(Z;c)$ near $\varphi_{Ray}$, as well as a viscous boundary layer correction $\phi_f(Z;c)$ around an exponential profile. The idea is to use the Rayleigh-Airy iteration used in \cite{GGN-DMJ,GMM-duke} to eliminate errors arising from these approximations. Compared to \cite{GGN-DMJ,GMM-duke}, in our problem the convergence of this iteration must be justified in the regime $\alpha,c\sim O(1)$, and in the presence of inflection points within the profile. Finally, we seek an Orr-Sommerfeld solution $\phi$ of the form 
\begin{align}
	\phi(Z;c)=\phi_{s}(Z;c)-\frac{\partial_Z\phi_s(0;c)}{\partial_Z\phi_f(0;c)}\phi_f(Z;c).\nonumber
\end{align}
The boundary conditions in \eqref{eq:OS-w} yield the following dispersion relation:
\begin{align}
	\phi(0;c)=\phi_s(0;c)-\frac{\partial_Z\phi_s(0;c)}{\partial_Z\phi_f(0;c)}\phi_f(0;c)=0.\nonumber
\end{align}
Thanks to the boundary layer structure of $\phi_f$, the factor
$$\frac{\partial_Z\phi_s(0;c)}{\partial_Z\phi_f(0;c)}\sim O(\nu^{\frac14}).
$$	
As a result, the viscous boundary layer does not affect the leading-order dispersion relation, so that an unstable eigenvalue can be constructed near $c_{Ray}$. This stands in sharp contrast to the two-dimensional Tollmien-Schlichting instability \cite{GGN-DMJ}, which is driven by the interaction between slow and fast modes. The detail will be given in Theorem \ref{thm-OSeq-solution}. \\

\item \underline{Step 3. Recover tangential velocity components.} Let $\Psi\triangleq \mathrm{i}\beta \tilde{u}-\mathrm{i}\sigma \tilde{v}$ denote the vorticity around the $Z$-axis. The tangential velocity components  $\tilde{u}$ and $\tilde{v}$ can be recovered in terms of the vertical component $\tilde{w}$ through the following relations:
\begin{equation}\label{eq:Airy-vorticity}
	\left\{
	\begin{aligned}
		&\mathrm{Airy}[\Psi]\triangleq\varepsilon(\partial_Z^2-\alpha^2)\Psi-(U-c)\Psi=(-\frac{\sigma}{\alpha}\partial_Zv_s+\frac{\beta}{\alpha}\partial_Z u_s)\tilde{w},\\
		&\Psi(0)=\Psi(\infty)=0,
	\end{aligned}
	\right.
\end{equation}
and
\begin{equation}
	\mathrm{i}\sigma\tilde{u}+\mathrm{i}\beta\tilde{v}=-\partial_Z\tilde{w}.\nonumber
\end{equation}
From these two equations, we can establish the $H^1$-bounds of $(\tilde{u},\tilde{v},\tilde{w})$. The detail will be given in Section \ref{sec-recover}.

\end{itemize}
\bigskip

The paper is organized as follows: Section \ref{sec2} elaborates on the structural conditions for $v_s$ and provides specific examples that satisfy these conditions. In Section \ref{sec-Ray}, we construct the homogeneous Rayleigh solution $\varphi_{Ray}$, and prove the inviscid instability. Section \ref{sec-OS} is devoted to constructing the homogeneous Orr-Sommerfeld solution $\phi_s$ near $\varphi_{Ray}$, as well as a viscous sublayer $\phi_f$, and to proving the viscous instability.  The recovery of tangential velocity components is discussed in Section \ref{sec-recover}. \\

{\bf Notations and convention:} Throughout the paper,  $C$ denotes a generic positive constant, which may vary from line to line.  We write $A=O(1)B$ to indicate that there exists a generic constant $C$ such that $A\leq CB$, and $A=o(1)B$ to indicate this constant $C$ be made arbitrarily small. For any $z\in \mathbb{C}\setminus \mathbb{R}_-$,  we take the principle analytic branch of $\log z$ and $z^k, k\in (0,1)$, i.e.
$$\log z\triangleq \log|z|+i\text{arg}z,~ z^k\triangleq|z|^ke^{ik \text{arg} z},~\text{arg}z\in (-\pi,\pi].$$
\section{Structural conditions for velocity profiles}\label{sec2}

In this section, we specify structural conditions on the spanwise boundary layer profile $v_s(Z)$. Let $v_s(Z)\in C^4(\overline{\mathbb{R}_+})$ such that
	\begin{align}
			v_s(0)=0,~ \lim\limits_{Z\to\infty}v_s(Z)=v_\infty, ~ \sum_{k=0}^{4}\sup_{Z\geq0}\left|\partial_Z^k(v_s(Z)-v_{\infty})e^{\eta_0 Z}\right|<+\infty,\nonumber
	\end{align}
for some constants $v_\infty\in \mathbb{R},$ and  $\eta_0>0$. We denote
	\begin{align*}
	\Upsilon_u=\partial_Z^2u_s(0)\partial_Zu_s(0),\quad \Upsilon_v=\partial_Z^2v_s(0)\partial_Zv_s(0),\quad \Upsilon_m=\partial_Z^2u_s(0)\partial_Zv_s(0)+\partial_Z^2v_s(0)\partial_Zu_s(0).
	\end{align*}
Recall the modified base flow $U(Z)$ defined in \eqref{U}.
\begin{proposition}\label{pSC}
	If $u_s$ and $v_s$ satisfy the independence condition 
	\begin{align}
		\partial_Z^2u_s(0)\partial_Zv_s(0)\neq \partial_Z^2v_s(0)\partial_Zu_s(0),\label{np}
	\end{align}
	 and 
	\begin{align}
	 \max\left\{\|\frac{v_s}{u_s}\|_{L^\infty},|\frac{\partial_Zv_s(0)}{\partial_Zu_s(0)}|\right\}< |\Upsilon_m|^{-1}\left(\sqrt{(\Upsilon_u-\Upsilon_v)^2+\Upsilon_m^2}+\Upsilon_u-\Upsilon_v\right),\label{s}
	\end{align}
then there exist constants $\lambda_1$ and $\lambda_2$ such that $U(Z)=\lambda_1 u_s(Z)+\lambda_2v_s(Z)$ satisfies  \eqref{SC}.
\end{proposition}
\begin{proof}
 Set $\Theta\triangleq\frac{\Upsilon_u-\Upsilon_v}{\sqrt{(\Upsilon_u-\Upsilon_v)^2+\Upsilon_m^2}}$.  Then we define
\begin{align}
	\lambda_1=\sqrt{\frac{1}{2}(1+\Theta)},~~	\lambda_2=\mathrm{sgn}(\Upsilon_{m})\sqrt{\frac{1}{2}(1-\Theta)}.\nonumber
\end{align}
 Using \eqref{s}, we deduce that for any $Z>0$, 
	\begin{align*}
		U(Z)\geq u_s\left(\lambda_1-|\lambda_2|\|\frac{v_s}{u_s}\|_{L^\infty}\right)> u_s
		\left(\sqrt{\frac12(1+\Theta)}-\sqrt{\frac{1}{2}(1-\Theta)}\frac{1+\Theta}{\sqrt{1-\Theta^2}}\right)=0.
	\end{align*}
In particular, $U(Z)>0$ for any $Z>0$. Similarly, $$ \partial_ZU(0)\geq \partial_Zu_s(0)\left(\lambda_1-|\lambda_2||\frac{v_s'(0)}{u_s'(0)}|\right)>0.
$$
Moreover, note that
	\begin{align*}
		\partial_ZU(0)\partial_Z^2U(0) &= \frac{1}{2}\left(\Upsilon_u+\Upsilon_v+\sqrt{(\Upsilon_u-\Upsilon_v)^2+\Upsilon_m^2}\right)\\
		&=\frac{1}{2}\left(\Upsilon_u+\Upsilon_v+\sqrt{(\Upsilon_u+\Upsilon_v)^2+
			\Big(\partial_Z^2u_s(0)\partial_Zv_s(0)-\partial_Z^2v_s(0)\partial_Zu_s(0)\Big)^2}\right).
	\end{align*}
Thanks to the independence condition \eqref{np}, we have $\partial_ZU(0)\partial_Z^2U(0)>0$, which implies 
	\begin{align*}
		& \partial_Z^2U(0)>0.
	\end{align*}
The proof of Proposition \ref{pSC} is complete.
\end{proof}
The simplest example of instability is given by the following exponential profile:
	\begin{align}
	\label{exp}
	u_s(Z)=1-e^{-k_{u}Z},~v_s(Z)=v_{\infty}(1-e^{-k_vZ})
\end{align}
with some constants $k_u,k_v>0$. We denote $k\triangleq\frac{k_v}{k_u}$ as the ratio of change rates of $v_s$ and $u_s$. 
\begin{lemma}\label{lmexp}
If $k>1$ and $k|v_{\infty}| \geq \max\{1,\sqrt{\frac{2}{k-1}}\}$, then the exponential profile \eqref{exp} satisfies \eqref{np} and \eqref{s}.
\end{lemma}
\begin{proof}
Note that $k>1$ implies \eqref{np}, and it is suffcient to show \eqref{s}. Note that $|\frac{v_s'(0)}{u_s'(0)}|=k|v_\infty|$ and
	$$\|\frac{v_s}{u_s}\|_{L^\infty}\leq \frac{|v_{\infty}|k_v}{k_u}\sup_{Z\geq 0}\frac{\int_{0}^Ze^{-k_vZ'}d Z'}{\int_{0}^Ze^{-k_uZ'}d Z'}\leq k |v_{\infty}|,$$
	where we have used $k>1$ in the first step. Moreover, a direct computation yields that
	$$\text{R.H.S. of \eqref{s} }>\frac{2(k^3v_{\infty}^2-1)}{k|v_\infty|(1+k)}\geq \frac{2k(kv_{\infty})^2-(k-1)(kv_{\infty})^2}{k|v_\infty|(1+k)}\geq k|v_{\infty}|\geq \text{L.H.S. of \eqref{s} }.
	$$
	The proof of Lemma \ref{lmexp} is complete.
\end{proof}
\begin{remark}
	This Lemma implies that if the change rate of the spanwise profile $v_s$ exceeds that of the streamwise profile $u_s$, then the three-dimensional shear flow becomes unstable.
\end{remark}
\section{The Rayleigh equation}\label{sec-Ray}

In this section, we construct a homogeneous solution to the following  Rayleigh equation:
\begin{align}\label{eq:Ray-homo}
	\left\{
	\begin{aligned}
		&\mathrm{Ray}[\varphi]=(U(Z)-c)(\partial_Z^2-\alpha^2)\varphi (Z)-U''(Z)\varphi (Z)=0,\quad Z>0,\\ &\varphi(\infty)=0,
	\end{aligned}
	\right.	
\end{align}
where $U(Z)$ is the new base flow defined in \eqref{U}. Thanks to the structural condition \eqref{SC}, we can find positive constants $Z_0$ and $M_0$, such that
\begin{align}\label{eq:Hyper-U}
	\partial_ZU(Z),~ \partial_Z^2U(Z)\geq M_0 ~\mbox{for}~0\leq Z\leq Z_0,~~\text{and }~~U(Z)\geq M_0 ~\mbox{for}~Z\geq Z_0.
\end{align}
Moreover, there exist positive constants $\eta_0$ and $U_\infty$ such that
\begin{align}\label{eq:Hyper-u-v}
\sum_{k=0}^{4}\sup_{Z\geq0}\left|\partial_Z^k(U(z)-U_{\infty})e^{\eta_0 Z}\right|<+\infty.
\end{align}
We solve \eqref{eq:Ray-homo} for parameters $(\alpha,c)\in \mathbb{H}_1$, where
\begin{align}\label{def-H1}
	\mathbb{H}_1=\{(\alpha,c)\in \mathbb{R}_{+}\times\mathbb{C}\mid~\alpha|\log c_i|\ll1,~|c|\ll 1,~ c_r, c_i>0\}.
\end{align}

\subsection{Inhomogeneous Rayleigh equation}
In this subsection, we establish the solvability of the following inhomogeneous Rayleigh equation:
\begin{align}\label{eq:Ray-force}
	\left\{
	\begin{aligned}
		&(U(Z)-c)(\partial_Z^2-\alpha^2)\varphi (Z)-U''(Z)\varphi (Z)=F(Z),\quad Z>0,\\ &\varphi(\infty)=\partial_Z\varphi(\infty)=0,
	\end{aligned}
	\right.	
\end{align}
where $F$ is a general source term. For any $\theta>0$, we define the weighted norm $$\|f\|_{L^\infty_\theta}=\sup_{Z\geq 0}e^{\theta Z}|f(Z)|.$$ Then the solution space is defined by
$
\Tilde{\mathcal{Y}}_\theta=\big\{f: \mathbb{R}_+\to\mathbb{C},\   \|f\|_{\Tilde{\mathcal{Y}}_\theta}<\infty\big\},
$
where
\begin{align*}
	\|f\|_{\Tilde{\mathcal{Y}}_\theta}=\|f\|_{L_\theta^\infty}+\|\partial_Z f\|_{L_\theta^\infty}+\|(U-c)\partial_Z^2f\|_{L_\theta^\infty}.
\end{align*}
We also define $\mathcal{Y}_\theta=\big\{f: \mathbb{R}_+\to\mathbb{C},\   \|f\|_{\mathcal{Y}_\theta}<\infty\big\},$
where
\begin{align*}
	\|f\|_{\mathcal{Y}_\theta}=	\|f\|_{\Tilde{\mathcal{Y}}_\theta}+
	\|(U-c)^{2}\partial_Z^3 f\|_{L_\theta^\infty}+\|(U-c)^{3}\partial_Z^4 f\|_{L_\theta^\infty}.
\end{align*}

\begin{proposition}\label{prop:ray-non2}
	For any $F\in L_\theta^{\infty}(\mathbb{R}_+)$, there exists a solution $\varphi_{non}\in \Tilde{\mathcal{Y}_\theta}$ to \eqref{eq:Ray-force} satisfying
	\begin{align}\label{eq-Ray-nonhom-tilde-Y}
		\|\varphi_{non}\|_{\Tilde{\mathcal{Y}}_\theta}\leq C|\log c_i|\|F\|_{L_\theta^\infty}.
	\end{align}
	In addition, if $\partial_Z^j F\in L_\theta^{\infty}(\mathbb{R}_+)$ for $j=1,2$, then $\varphi_{non}\in\mathcal{Y}_\theta$, and it holds that
	\begin{align}\label{eq-Ray-nonhom-Y}
		\|\varphi_{non}\|_{\mathcal{Y}_\theta}\leq C\left(|\log c_i|\|F\|_{L_\theta^\infty}+\|(U-c)\partial_ZF\|_{L_\theta^\infty}
		+\|(U-c)^2\partial_Z^2F\|_{L_\theta^\infty}\right).
	\end{align}
\end{proposition}
\begin{proof} Note that 
	$$\mathrm{Ray}[\varphi]=\partial_Z\left(\big(U-c\big)^2\partial_Z\left(\frac{\varphi}{U-c}\right)\right)-\alpha^2\big(U-c\big)\varphi.
	$$
	We then
	reformulate the Rayleigh equation into the following integral form:
	\begin{align}\label{bc1.1}
		\varphi(Z)=&\alpha^2\big(U(Z)-c\big)\int_Z^\infty \frac{1}{\big(U(Z')-c\big)^2} dZ'\int_{Z'}^\infty \big(U(Z'')-c\big)\varphi(Z'')dZ''\nonumber\\
		&+\big(U(Z)-c\big)\int_{Z}^\infty \frac{1}{\big(U(Z')-c\big)^2}\int_{Z'}^\infty F(Z'') dZ'':=H_1+H_2.
	\end{align}
Let $Z_0$ be the number given in \eqref{eq:Hyper-U}. We split the analysis into the following two cases:
	
	Case 1: $Z\geq Z_0$. Since $|c|\ll1$, we have
	$$|U-c|\geq C^{-1}>0 \text{ for } Z\geq Z_0.$$ Thus,
	 we deduce that
	\begin{align}
		|H_1|\leq C\alpha^2e^{-\theta Z}\|\varphi\|_{L^\infty_\theta},~~|H_2|\leq Ce^{-\theta Z}\|F\|_{L^\infty_{\theta}}.\nonumber
	\end{align}

	Case 2: $Z\leq Z_0$. We split $H_1$ into two parts:
	$$
	\begin{aligned}
		H_1=&\alpha^2\big(U(Z)-c\big)\int_{Z_0}^\infty \frac{1}{\big(U(Z')-c\big)^2} \int_{Z'}^\infty \big(U(Z'')-c\big)\varphi(Z'')dZ''dZ'\\
		&+\alpha^2\big(U(Z)-c\big)\int_Z^{Z_0}\frac{1}{\big(U(Z')-c\big)^2} \int_{Z'}^\infty \big(U(Z'')-c\big)\varphi(Z'')dZ''dZ'\\
		:=&H_{11}+H_{12}.
	\end{aligned}
	$$
	Note that $|H_{11}|\leq C\alpha^2 \|\varphi\|_{L^\infty_{\theta}}$. For $H_{12}$, we denote $G_1(Z)\triangleq\int_{Z}^\infty\big(U(Z')-c\big)\varphi(Z')d Z'.$ Let $0<Z_c\ll Z_0$ be the point such that $U(Z_c)=c_r.$ Then,
	integration by parts yields
	\begin{align}
		\int_{Z}^{Z_0}\frac{G_1(Z')}{\big(U(Z')-c\big)^2}d Z'&=\frac{1}{U'(Z_c)}\left(\int_Z^{Z_0}\frac{G_1(Z')U'(Z')}{\big(U(Z')-c\big)^2} dZ'+\int_{Z}^{Z_0}\frac{G_1(Z')(U'(Z_c)-U'(Z'))}{\big(U(Z')-c\big)^2}\right)d Z'\nonumber\\
		&=\frac{G_1(Z)}{U'(Z_c)\big(U(Z)-c\big)}+O(1)\|G_1\|_{W^{1,\infty}}\left(1+\int_Z^{Z_0}\frac{1}{|U(Z')-c|}d Z'\right)\nonumber\\
		&=\frac{G_1(Z)}{U'(Z_c)\big(U(Z)-c\big)}+O(1)|\log c_i|\|G_1\|_{W^{1,\infty}},\label{g}
	\end{align}
which implies
	\begin{align}
		|H_{12}|\leq C \alpha^2 |\log c_i| \|G_1\|_{W^{1,\infty}}\leq C\alpha^2|\log c_i|\|\varphi\|_{L^\infty_{\theta}}.\nonumber
	\end{align}
Hence, we have $$|H_1|\leq C\alpha^2|\log c_i|\|\varphi\|_{L^\infty_{\theta}}.$$ Similar to $H_1$, we denote $G_2(Z)\triangleq\int_Z^\infty F(Z')d Z'$. Then, it holds that 
	$$|H_2|\leq C|\log c_i|\|G_2\|_{W^{1,\infty}}\leq C|\log c_i|\|F\|_{L^\infty_{\theta}}.
	$$
	Substituting bounds on $H_1$ and $H_1$ into \eqref{bc1.1} yields 
	\begin{align}
		\|\varphi\|_{L^\infty_\theta}\leq C\alpha^2|\log c_i|\|\varphi\|_{L^\infty_{\theta}}+C|\log c_i|\|F\|_{L^\infty_{\theta}}.\nonumber
	\end{align}
Note that $\alpha^2|\log c_i|\ll 1$ due to \eqref{def-H1}. Therefore, the $L^\infty_{\theta}$ solution $\varphi$ to \eqref{bc1.1} can be constructed by a fixed point argument. Moreover, it satisfies 
	\begin{align}
		\|\varphi\|_{L^\infty_\theta}\leq C|\log c_i|\|F\|_{L^\infty_{\theta}}.\label{bc1.2}
	\end{align}
	
	Next we estimate $\partial_Z\varphi$. Taking derivative in \eqref{bc1.1} yields
	\begin{align}
		\partial_Z\varphi(Z)=&
		\alpha^2\left(U'(Z)\int_Z^\infty \frac{G_1(Z')}{\big(U(Z')-c\big)^2} dZ'-\frac{G_1(Z)}{U(Z)-c}\right)\nonumber\\
		&+U'(Z)\int_{Z}^\infty \frac{G_2(Z)}{\big(U(Z')-c\big)^2}dZ'-\frac{G_2(Z)}{U(Z)-c}.\nonumber
	\end{align}
	For $Z\geq Z_0$, we use $|U(Z)-c|\geq C^{-1}$ to obtain
	$$|\partial_Z\varphi(Z)|\leq C\alpha^2e^{-\theta Z}\|\varphi\|_{L^\infty_{\theta}}+Ce^{-\theta Z}\|F\|_{L^\infty_{\theta}}.
	$$ 
	For $Z\leq Z_0$, we split the integral domain into $(Z,Z_0)$ and $(Z_0,\infty)$ as before. In the interval $(Z_0,\infty)$, it holds 
	$$\left|U'(Z)\int_{Z_0}^\infty \frac{G_i(Z')}{\big(U(Z')-c\big)^2} dZ'\right|\leq C\|G_i\|_{L^\infty_{\theta}},~i=1,2.
	$$
On the other hand, thanks to \eqref{g}, it holds that
	$$\begin{aligned}
		\left| U'(Z) \int_{Z}^{Z_0} \frac{G_i(Z')}{\big(U(Z')-c\big)^2} dZ' -\frac{G_i(Z)}{U(Z)-c}\right|&\leq C\left|\frac{U'(Z)}{U'(Z_c)}-1\right|\left|\frac{G_i(Z)}{U(Z)-c}\right|+C|\log c_i|\|G_i\|_{W^{1,\infty}}\\
		&\leq C|\log c_i|\|G_i\|_{W^{1,\infty}},~i=1,2.
	\end{aligned}
	$$
Consequently, we obtain
	$$|\partial_Z\varphi(Z)|\leq C\alpha^2|\log c_i|\|\varphi\|_{L^\infty_{\theta}}+C|\log c_i|\|F\|_{L^\infty_{\theta}}.
	$$
	Combining these two cases yields
	\begin{align}
		\|\partial_Z\varphi\|_{L^\infty_\theta}\leq C\alpha^2|\log c_i|\|\varphi\|_{L^\infty_{\theta}}+C|\log c_i|\|F\|_{L^\infty_{\theta}}\leq C|\log c_i|\|F\|_{L^\infty_{\theta}},\nonumber
	\end{align}
where we have used \eqref{bc1.2} in the last step. Estimates for higher-order derivatives can be obtained by using equation \eqref{eq:Ray-force}. The proof of Proposition \ref{prop:ray-non2} is complete.
\end{proof}

\subsection{Approximate homogeneous Rayleigh solution}
We introduce 
\begin{align}\label{eq:ray-main-def}
	\varphi_{Ray}^{(0)}(Z)\triangleq2\alpha e^{\alpha Z}\big(U(Z)-c\big)\int_Z^{+\infty}\frac{e^{-2\alpha Z'}}{(U(Z')-c)^2}dZ'.
\end{align}
It is straightforward to check 
\begin{align}\label{eq:airy-0}
	\mathrm{Ray}[\varphi_{Ray}^{(0)}]=2\alpha U'\varphi_{Ray}^{(0)}.
\end{align}
Therefore, $\varphi_{Ray}^{(0)}$ is an approximate solution to the homogeneous Rayleigh equation \eqref{eq:Ray-homo} at leading order. 
\begin{lemma}\label{lem:ray-main}
	Let $(\alpha, c)\in \mathbb{H}_1$. Then $\varphi_{Ray}^{(0)}$  satisfies the following pointwise estimate:
	\begin{itemize}
		\item[(a)] For $Z\geq Z_0$, it holds that
		\begin{align}
			\left|\varphi_{Ray}^{(0)}(Z)-\frac{e^{-\alpha Z}\big(U(Z)-c\big)}{(U_\infty-c)^2}\right|+\left|\partial_Z\varphi_{Ray}^{(0)}(Z)-\partial_Z\left(\frac{e^{-\alpha Z}\big(U(Z)-c\big)}{(U_\infty-c)^2}\right) \right|\leq C\alpha e^{-\eta_0 Z},\label{p1}
		\end{align}
	where $\eta_0$ is given in \eqref{eq:Hyper-u-v}.
		\item[(b)] For $0\leq Z\leq Z_0$, it holds that
		\begin{align}
			\left|\varphi_{Ray}^{(0)}(Z)-\frac{e^{-\alpha Z}\big(U(Z)-c\big)}{(U_\infty-c)^2}-\frac{2\alpha e^{-\alpha Z}}{U'(Z_c)} \right|&\leq C\alpha|U(Z)-c||\log c_i|,\label{p2}\\
			\left|\partial_Z\varphi_{Ray}^{(0)}(Z)-\frac{e^{-\alpha Z}U'(Z)}{(U_\infty-c)^2} \right|&\leq C\alpha|\log c_i|.\label{p3}
		\end{align}
	\end{itemize} 
	
	In particular, the boundary values of $\varphi_{Ray}^{(0)}(Z)$ have the following asymptotic expansions:
	\begin{align}
		\varphi_{Ray}^{(0)}(0)&=-\frac{c}{U_\infty^2}+\frac{2\alpha }{U'(Z_c)} +\mathcal{O}(1)(\alpha+|c|)|c||\log c_i|,\label{bc1}\\
		\partial_Z\varphi_{Ray}^{(0)}(0)&=\frac{U'(0)}{U_\infty^2}+\mathcal{O}(1)(\alpha+|c|)|\log c_i|.\label{bc2}
	\end{align}
	
	Moreover, 
	\begin{align}\label{eq-lem-ray-sub-profile}
		\big\|\varphi_{Ray}^{(0)}-(U_\infty-c)^{-2}(U-c)e^{-\alpha Z}\big\|_{\mathcal{Y}_{\eta_0}}\leq C\alpha|\log c_i|^2.
	\end{align}
\end{lemma}

\begin{proof} First, we prove (a). We split the integral in the definition \eqref{eq:ray-main-def} as follows:
	\begin{align}\label{kernel-large}\begin{aligned}
			\int_{Z}^{+\infty}\frac{e^{-2\alpha Z'}dZ'}{(U(Z')-c)^2}&=\frac{e^{-2\alpha Z}}{2\alpha(U_{\infty}-c)^2}+\int_{Z}^{+\infty}e^{-2\alpha Z'}\left(\frac{1}{(U(Z')-c)^2}-\frac{1}{(U_\infty-c)^2}
			\right)dZ'\\
			&=\frac{e^{-2\alpha Z}}{2\alpha(U_{\infty}-c)^2}+\int_Z^\infty\frac{e^{-2\alpha Z'}(U_\infty-U(Z'))(U(Z')+U_\infty-2c)}{(U_\infty-c)^2(U(Z')-c)^2}d Z'\\
			&\triangleq\frac{e^{-2\alpha Z}}{2\alpha(U_{\infty}-c)^2}+A(Z).
	\end{aligned}\end{align}
	Using the fact that $
		 |U(Z)-c|\sim 1$  for $Z\geq Z_0$,	and decay property $U(Z)-U_{\infty}\in L^\infty_{\eta_0}$, we can bound
	\begin{align}
		|A(Z)|\leq  C\|U_\infty-U(Z)\|_{L^\infty_{\eta_0}}\int_{Z}^{+\infty}e^{-(2\alpha+\eta_0)Z'}dZ'
		\leq Ce^{-(2\alpha+\eta_0) Z}.\label{A}
	\end{align}
	Consequently, it follows that 
	\begin{align}\label{phi0-large}
		\begin{aligned}
			\varphi_{Ray}^{(0)}(Z)&=2\alpha e^{\alpha Z}(U(Z)-c)\int_{Z}^{+\infty}\frac{e^{-2\alpha Z'}dZ'}{(U(Z')-c)^2}
			=\frac{e^{-\alpha Z}(U(Z)-c)}{(U_\infty-c)^2}+\mathcal{O}(\alpha) e^{-\eta_0 Z}.
	\end{aligned}\end{align}
	
	Next, taking derivative in \eqref{eq:ray-main-def} yields
	\begin{align}\label{phi0z}
		\partial_Z\varphi_{Ray}^{(0)}(Z)= & (\alpha+\frac{U'(Z)}{U(Z)-c})\varphi_{Ray}^{(0)}(Z)-\frac{2\alpha e^{-\alpha Z}}{U(Z)-c}.
	\end{align}
	Substituting \eqref{phi0-large} into \eqref{phi0z}, we deduce that
	\begin{align}\begin{aligned}
			\partial_Z\varphi_{Ray}^{(0)}(Z)= & \left(\alpha+\frac{U'(Z)}{U(Z)-c}\right)\frac{e^{-\alpha Z}(U(Z)-c)}{(U_\infty-c)^2}-\frac{2\alpha e^{-\alpha Z}}{U(Z)-c}+\mathcal{O}(\alpha) e^{-\eta_0 Z} \\
			=& \frac{\partial_Z\big(e^{-\alpha Z}(U(Z)-c)\big)}{(U_\infty-c)^2}+\mathcal{O}(\alpha) e^{-\eta_0 Z}.\nonumber
	\end{aligned}\end{align}
	The proof of \eqref{p1} is complete.\\
	
	Next we estabilish (b). For $0\leq Z\leq Z_0,$ we split $\varphi^{(0)}_{Ray}(Z)$ as
	\begin{align}\label{phi0-small}\begin{aligned}
			\varphi_{Ray}^{(0)}(Z)=2\alpha e^{\alpha Z}(U(Z)-c)\left(\int_{Z_0}^{+\infty}\frac{e^{-2\alpha Z'}dZ'}{(U(Z')-c)^2}+\int_{Z}^{Z_0}\frac{e^{-2\alpha Z'}dZ'}{(U(Z')-c)^2}\right).
	\end{aligned}\end{align}
	We need to further decompose the second integral in the bracket. Recall that $0<Z_c\ll Z_0$ is the point where $U(Z_c)=c_r$. Note that
	\begin{align}
	1&=\frac{1}{U'(Z_c)}\left[U'(Z)\left(1-\frac{U''(Z_c)}{U'(Z_c)^2}(U(Z)-c_r)\right)-g(Z)\right],\nonumber\\
	&=\frac{1}{U'(Z_c)}\left[U'(Z)\left(\gamma-\frac{U''(Z_c)}{U'(Z_c)^2}(U(Z)-c)\right)-g(Z)\right],
	\label{d}
	\end{align}
where
\begin{align}
	\gamma= 1-\frac{ic_iU''(Z_c)}{U'(Z_c)^2},\label{gamma}
\end{align}
and
	\begin{align}\label{def-g}\begin{aligned}
			g(Z)=&U'(Z)-U'(Z_c)-\frac{U''(Z_c)}{U'(Z_c)^2}U'(Z)(U(Z)-c_r).
	\end{aligned}\end{align}
	It is straightforward to check that 
	\begin{align}\label{bound-g}
		&g(Z)\in\mathbb{R},\quad\text{ and } \left\|\frac{g(Z)}{(U(Z)-c_r)^2}\right\|_{L^\infty}\lesssim 1. 
	\end{align}
Substituting \eqref{d} into the integral $\int_{Z}^{Z_0}\frac{e^{-2\alpha Z'}dZ'}{(U(Z')-c)^2}$, then integrating by parts, we obtain 
	\begin{align*}
			&\int_{Z}^{Z_0}\frac{e^{-2\alpha Z'}dZ'}{(U(Z')-c)^2}
			=\int_{Z}^{Z_0}\frac{e^{-2\alpha Z'}}{(U(Z')-c)^2}\frac{g(Z')-U'(Z')\left(\gamma-\frac{U''(Z_c)}{U'(Z_c)^2}(U(Z')-c)\right)}{-U'(Z_c)}dZ'\\
			&\qquad=-\int_{Z}^{Z_0}\frac{e^{-2\alpha Z'}g(Z')dZ'}{U'(Z_c)(U(Z')-c)^2}-\frac{e^{-2\alpha Z_0}}{U'(Z_c)}\left(\frac{\gamma}{U(Z_0)-c}+\frac{U''(Z_c)}{U'(Z_c)^2}\log(U(Z_0)-c)\right) \\
			&\qquad\quad+\frac{e^{-2\alpha Z}}{U'(Z_c)}\left(\frac{\gamma}{U(Z)-c}+\frac{U''(Z_c)}{U'(Z_c)^2}\log(U(Z)-c)\right)\\
			&\qquad\quad-\frac{2\alpha\gamma}{U'(Z_c)}\int_{Z}^{Z_0}\frac{e^{-2\alpha Z'}dZ'}{U(Z')-c}-\frac{2\alpha U''(Z_c)}{U'(Z_c)^3}\int_{Z}^{Z_0}e^{-2\alpha Z'}\log(U(Z')-c)dZ'.
	\end{align*}
	Therefore, we substitute the above equality back into \eqref{phi0-small} to obtain the following decomposition
	\begin{align}\label{phi0-small-decom}\begin{aligned}
			\varphi_{Ray}^{(0)}(Z)=~\sum_{i=1}^{5}I_i(Z),
	\end{aligned}\end{align}
 where
	\begin{align}\begin{aligned}
			I_1(Z)=&2\alpha e^{\alpha Z}(U(Z)-c)\int_{Z_0}^{+\infty}\frac{e^{-2\alpha Z'}dZ'}{(U(Z')-c)^2},\\
			I_{2}(Z) =& -2\alpha(U(Z)-c)\frac{e^{\alpha(Z-2 Z_0)}}{U'(Z_c)}\left(\frac{\gamma}{U(Z_0)-c}+\frac{U''(Z_c)}{U'(Z_c)^2}\log(U(Z_0)-c)\right), \\
			I_{3}(Z) =& \frac{2\alpha e^{-\alpha Z}}{U'(Z_c)}\left(\gamma+\frac{U''(Z_c)}{U'(Z_c)^2}(U(Z)-c)\log(U(Z)-c)\right), \\
			I_{4}(Z) =&   -\frac{2\alpha e^{\alpha Z}}{U'(Z_c)}(U(Z)-c)\int_{Z}^{Z_0}e^{-2\alpha Z'}\frac{g(Z')}{(U(Z')-c)^2}dZ',\\
			I_{5}(Z) =&   -\frac{4\alpha^2 e^{\alpha Z}}{U'(Z_c)}(U(Z)-c)\int_{Z}^{Z_0}e^{-2\alpha Z'}\left(\frac{\gamma}{U(Z')-c}+\frac{ U''(Z_c)}{U'(Z_c)^2}\log(U(Z')-c)\right)dZ'.\nonumber
	\end{aligned}\end{align}
Now we estimate $I_1$ to $I_5$ term by term. For $I_1$, we use \eqref{kernel-large} and \eqref{A} to obtain
	\begin{align}\label{I1}\begin{aligned}
			I_1(Z)&= 
			\frac{e^{-\alpha Z}(U(Z)-c)}{ (U_\infty-c)^2}+2\alpha e^{\alpha Z}(U(Z)-c)\int_{Z_0}^Z \frac{e^{-2\alpha Z'}}{(U_{\infty}-c)^2} dZ'+2\alpha e^{\alpha Z}(U-c)A(Z_0)\\
			&=\frac{e^{-\alpha Z}(U(Z)-c)}{ (U_\infty-c)^2}+\mathcal{O}(1)\alpha|U(Z)-c|.
	\end{aligned}\end{align}
The following bounds for $I_2$ and $I_3$ can be directly obtained:
	\begin{align}\label{I2-3}\begin{aligned}
			I_{2} (Z)=& \mathcal{O}(1)\alpha |U(Z)-c|, \\
			I_{3}(Z)= &~ \frac{2\alpha e^{-\alpha Z}}{U'(Z_c)}+\mathcal{O}(1)\alpha\left(c_i+\big|(U(Z)-c)\log(U(Z)-c)\big|\right)\\
			=&~\frac{2\alpha e^{-\alpha Z}}{U'(Z_c)}+\mathcal{O}(1)\alpha|U(Z)-c||\log c_i|.
	\end{aligned}\end{align}
	By virtue of \eqref{bound-g}, we obtain that
	\begin{align}\label{I4}
		I_4(Z) =& \mathcal{O}(1)\alpha|U(Z)-c|.
	\end{align}
Thanks to
	\begin{align}
		\int_{Z}^{Z_0}\frac{dZ'}{|U(Z')-c|} \lesssim |\log c_i|,\text{ and } 
		\int_{Z}^{Z_0}|\log(U(Z')-c)|\lesssim 1,\nonumber
	\end{align}
we deduce that
	\begin{align}\label{I5}
		I_5(Z)=~&\mathcal{O}(1)\alpha^2 |U(Z)-c|(|\log c_i|+1)=\mathcal{O}(1)\alpha^2 |U(Z)-c||\log c_i|.
	\end{align}
By substituting \eqref{I1}-\eqref{I5} into \eqref{phi0-small-decom}, we get
	\begin{align}\label{phi0-small-exp}
		\begin{aligned}
			\varphi_{Ray}^{(0)}(Z)
			=~&\frac{e^{-\alpha Z}(U(Z)-c)}{ (U_\infty-c)^2}+\frac{2\alpha e^{-\alpha Z}}{U'(Z_c)}+\mathcal{O}(1)\alpha|U(Z)-c||\log c_i|,
	\end{aligned}\end{align}
	which is precisely \eqref{p2}. Moreover, substituting \eqref{phi0-small-exp} into \eqref{phi0z} yields
	\begin{align}\label{phi0z-small}
		\begin{aligned}
			\partial_Z\varphi_{Ray}^{(0)}(Z)=&  \left(\alpha+\frac{U'(Z)}{U(Z)-c}\right)\left(\frac{e^{-\alpha Z}(U(Z)-c)}{ (U_\infty-c)^2}+\frac{2\alpha e^{-\alpha Z}}{U'(Z_c)}\right)-\frac{2\alpha e^{-\alpha Z}}{U(Z)-c}\\
			&+\mathcal{O}(1)\alpha|\log c_i||U(Z)-c|\left|\alpha+\frac{U'(Z)}{U(Z)-c}\right|\\
			=& \frac{e^{-\alpha Z}U'(Z)}{(U_\infty-c)^2}+\mathcal{O}(1)\alpha|\log c_i|,
		\end{aligned}
	\end{align}
	which is precisely \eqref{p3}.

	By evaluating \eqref{phi0-small-exp} and \eqref{phi0z-small} at $Z=0$ respectively, we obtain
	\begin{align}\label{phi0-0-exp}\begin{aligned}
			\varphi_{Ray}^{(0)}(0) =& -\frac{c}{(U_\infty-c)^2}+\frac{2\alpha}{U'(Z_c)}
			+\mathcal{O}(1)\alpha|c\log c_i|\\
			=&-\frac{c}{U_\infty^2}+\frac{2\alpha}{U'(Z_c)}
			+\mathcal{O}(1)(|c|^2+\alpha|c\log c_i|), 
	\end{aligned}\end{align}
	and \begin{align*}
		\partial_Z\varphi_{Ray}^{(0)}(0)= & \frac{U'(0)}{(U_\infty-c)^2}+\mathcal{O}(1)\alpha|\log c_i|=\frac{U'(0)}{U_\infty^2}+\mathcal{O}(1)(|c|+\alpha|\log c_i|).
	\end{align*}
	The proof of \eqref{bc1} and \eqref{bc2} is complete.
	
	Finally, we introduce 
	\begin{align*}
		\tilde{\varphi}^{(0)}(Z)\triangleq\varphi_{Ray}^{(0)}(Z)-(U_\infty-c)^{-2}\big(U(Z)-c\big)e^{-\alpha Z}.
	\end{align*}
Note that $\tilde{\varphi}^{(0)}$ satisfies
	\begin{align}\label{eq:airy-1}
		\mathrm{Ray}[\tilde{\varphi}^{(0)}]=2\alpha U'\big(\varphi_{Ray}^{(0)}+(U_\infty-c)^{-2}\big(U-c\big)e^{-\alpha Z}\big):=\tilde{F}.
	\end{align}
	Using \eqref{eq:airy-0} and the bounds \eqref{p1}, \eqref{p2} for $\varphi_{Ray}^{(0)}$, we deduce that $\big\|(U-c)\partial_Z^2\varphi_{Ray}^{(0)}\big\|_{L^\infty}\leq C|\log c_i|,$ which implies
	\begin{align*}
		\|\tilde{F}\|_{L_{\eta_0}^\infty}+ \|\partial_Z\tilde{F}\|_{L_{\eta_0}^\infty}
		+\|(U-c)\partial_Z^2\tilde{F}\|_{L_{\eta_0}^\infty}\leq C\alpha |\log c_i|. 
	\end{align*}
	Applying the Proposition \ref{prop:ray-non2} to \eqref{eq:airy-1}, we obtain the bound  \eqref{eq-lem-ray-sub-profile}, and thus proof of Lemma \ref{lem:ray-main} is complete.
\end{proof}

We establish the bound on $\partial_c\varphi^{(0)}_{Ray}(0;c)$ in the following
\begin{lemma}\label{lmphic}
	Let $(\alpha,c)\in \mathbb{H}_1$. Then
	\begin{align}
		\|\partial_c\varphi_{Ray}^{(0)}\|_{L^\infty}\leq C.\label{phic}
	\end{align} 
In particular, $\varphi^{(0)}_{Ray}(0;c)$ is analytic in $c$, and 
\begin{align}
	\partial_c\varphi^{(0)}_{Ray}(0;c)=-\frac{1}{U_{\infty}^2}+O(1)(\alpha|\log c_i|+|c|).\label{dphic}
\end{align}
\end{lemma}
\begin{proof}
	Differentiating \eqref{eq:ray-main-def} with respect to $c$ yields
\begin{equation}
	\begin{aligned}
	\partial_c\varphi^{(0)}_{Ray}(Z)&=-2\alpha e^{\alpha Z}\int_Z^\infty \frac{e^{-2\alpha Z'}}{(U(Z')-c)^2}d Z'+4\alpha e^{\alpha Z}(U(Z)-c)\int_{Z}^\infty \frac{e^{-2\alpha Z'}}{(U(Z')-c)^3} dZ'.\nonumber
\end{aligned}
\end{equation}
For $Z\geq Z_0$, it is straightforward to show that $|\partial_c\varphi^{(0)}_{Ray}(Z)|\leq C.$ Thus, we only need to consider the case $Z\leq Z_0$. Similar to before, we decompose
\begin{align}
		\partial_c\varphi^{(0)}_{Ray}(Z)=I_6+I_7+I_8+I_9+I_{10},\label{phi1c}
\end{align}
where
\begin{equation}\nonumber
\begin{aligned}
	I_6&=\frac{-e^{-\alpha Z}}{(U_\infty-c)^2}+\frac{2(U(Z)-c)e^{-\alpha Z}}{(U_\infty-c)^3},\\
	I_7&=-2\alpha e^{\alpha Z}\int_{Z_0}^\infty \frac{e^{-2\alpha Z'}}{(U(Z')-c)^2} -\frac{e^{-2\alpha Z'}}{(U_\infty-c)^2}dZ'\\
	&\qquad+4\alpha (U(Z)-c)e^{\alpha Z}\int_{Z_0}^\infty \frac{e^{-2\alpha Z'}}{(U(Z')-c)^3}-\frac{e^{-2\alpha Z'}}{(U_\infty-c)^3} dZ',\\
	I_8&=\frac{2\alpha e^{\alpha Z}}{(U_\infty-c)^2}\int_{Z}^{Z_0}e^{-2\alpha Z'}d Z'-\frac{4\alpha(U(Z)-c)e^{\alpha Z}}{(U_\infty-c)^3}\int_{Z}^{Z_0}e^{-2\alpha Z'}d Z',\\
	I_9&=-2\alpha e^{\alpha Z}\int_{Z}^{Z_0}\frac{e^{-2\alpha Z'}}{(U(Z')-c)^2} dZ',\\
	I_{10}&=4\alpha e^{\alpha Z}(U(Z)-c)\int_{Z}^{Z_0}\frac{e^{-2\alpha Z'}}{(U(Z')-c)^3} dZ'.
\end{aligned}	
\end{equation}

Now we estimate $I_6$ to $I_{10}$ term by term. Using the decay property $\|U-U_\infty\|_{L^\infty_{\eta_0}}\leq C$, and $|U(Z')-c|\gtrsim 1$ when $Z'\geq Z_0$, we deduce that
\begin{align}\label{I7}
	|I_7|\leq C\alpha.
\end{align}
Note that $|\int_{Z}^{Z_0}e^{-2\alpha Z'} dZ'|\leq Z_0\lesssim 1$ and $e^{\alpha Z}\leq e^{\alpha Z_0}\lesssim 1$. Thus, we have
\begin{align}\label{I8}
	|I_8|\leq C\alpha.
\end{align}

To estimate $I_9$, we use the decomposition \eqref{d} to get
\begin{align}
	I_9&=\frac{-2\alpha e^{\alpha Z}}{U'(Z_c)}\int_Z^{Z_0}e^{-2\alpha Z'}\left( \frac{\gamma U'(Z')}{(U(Z')-c)^2}-\frac{U''(Z_c)U'(Z')}{U'(Z_c)^2(U(Z')-c)}-\frac{g(Z')}{(U(Z')-c)^2}\right) dZ'\nonumber\\
	&:=I_{91}+I_{92}+I_{93}.\label{I9d}
\end{align}
Integrating by parts yields
\begin{align}\label{I91}
	I_{91}&=\frac{-2\alpha\gamma e^{-\alpha Z}}{U'(Z_c)(U(Z)-c)}+\frac{2\alpha\gamma e^{\alpha Z-2\alpha Z_0}}{U'(Z_c)(U(Z_0)-c)}+\frac{4\alpha^2\gamma e^{\alpha Z}}{U'(Z_c)}\int_Z^{Z_0}\frac{e^{-2\alpha Z'}}{U(Z')-c} dZ'\nonumber\\
	&=\frac{-2\alpha\gamma e^{-\alpha Z}}{U'(Z_c)(U(Z)-c)}+O(\alpha),
\end{align}
where we have used $\int_{Z}^{Z_0}\frac{1}{|U(Z')-c|}dZ'\leq |\log c_i|$ and $\alpha |\log c_i|\ll 1$ in the last step. Similarly,
\begin{align}\label{I92}
	I_{92}=&\frac{-2\alpha U''(Z_c) e^{-\alpha Z}}{U'(Z_c)^3}\log (U(Z)-c)+\frac{2\alpha U''(Z_c) e^{\alpha Z-2\alpha Z_0}}{U'(Z_c)^3}\log (U(Z_0)-c)\nonumber\\
	&-\frac{4\alpha^2U''(Z_c)e^{\alpha Z}}{U'(Z_c)^3}\int_{Z}^{Z_0}e^{-2\alpha Z'}\log(U(Z')-c)d Z'\nonumber\\
	=&O(1)\alpha|\log c_i|.
\end{align}
By using \eqref{bound-g}, we obtain that
\begin{align}
	I_{93}\leq C\alpha.\label{I93}
\end{align}
Substituting \eqref{I91}-\eqref{I93} into \eqref{I9d} yields
\begin{align}
	I_9=\frac{-2\alpha\gamma e^{-\alpha Z}}{U'(Z_c)(U(Z)-c)}+O(1)\alpha|\log c_i|.\label{I9}
\end{align}
Similarly, we can obtain the following bound for $I_{10}$:
\begin{align}
	I_{10}=\frac{2\alpha\gamma e^{-\alpha Z}}{U'(Z_c)(U(Z)-c)}+O(1)\alpha|\log c_i|\label{I10}.
\end{align}

Substituting \eqref{I7}, \eqref{I8}, \eqref{I9} and \eqref{I10} into \eqref{phi1c}, we get
\begin{align}\label{phic1}
	\partial_c\varphi_{Ray}^{(0)}(Z)=\frac{-e^{-\alpha Z}}{(U_\infty-c)^2}+\frac{2(U(Z)-c)e^{-\alpha Z}}{(U_\infty-c)^3}+O(1)\alpha|\log c_i|,
\end{align}
which implies \eqref{phic}. Evaluating \eqref{phic1} at $Z=0$ yields \eqref{dphic}. The analyticity follows from the explicit expression \eqref{eq:ray-main-def}. This completes the proof of Lemma \ref{lmphic}.
\end{proof}

The asymptotic value of $\varphi^{(0)}_{Ray}(0)$ obtained in \eqref{bc1} is not sufficient to capture the inviscid instability. The following lemma establishes a more accurate second-order asymptotic expansion of $\varphi_{Ray}^{(0)}(0)$ and $\partial_Z\varphi_{Ray}^{(0)}(0)$ with respect to the small parameters $\alpha$ and $c$.

\begin{lemma}\label{lem:ray-main-asy}
	Let $(\alpha, c)\in\mathbb{H}_1$. If, in addition, $c_i\ll \min\{\alpha, c_r\}$, then
	\begin{align}
		\mathrm{Re}\left(\varphi_{Ray}^{(0)}(0)\right)= & -\frac{c_r}{U_\infty^2}+ \frac{2\alpha}{U'(Z_c)}\left(1-\frac{U''(Z_c)}{U'(Z_c)^2}c_r\log|c|\right) +\mathcal{O}\big((\alpha+c_r)c_r\big),\label{Re1} \\
		\mathrm{Im}\left(\varphi_{Ray}^{(0)}(0)\right)= &  -\frac{c_i}{U_\infty^2}- \frac{2\alpha U''(Z_c)}{U'(Z_c)^3}\Big(c_r\arg(-c)+c_i\log|c|\Big)
		+o\big((\alpha+c_r)c_r\big),\label{Im1}\\
		\mathrm{Re}\left(\partial_Z\varphi_{Ray}^{(0)}(0)\right)=&\frac{U'(0)}{U_\infty^2}+\frac{2U''(Z_c)}{U'(Z_c)^2}\alpha\log|c|
		+\mathcal{O}(\alpha+c_r),\label{Re2}\\
		\mathrm{Im} \left(\partial_Z\varphi_{Ray}^{(0)}(0)\right)=&\frac{2U''(Z_c)}{U'(Z_c)^2}\alpha\arg(-c)
		+o(\alpha+c_r).\label{Im2}
	\end{align}
\end{lemma}
\begin{proof}
	\underline{\bf Estimates of $\varphi_{Ray}^{(0)}(0)$}. We start from the decomposition \eqref{phi0-small-decom}. Evaluating \eqref{phi0-small-decom} at $Z=0$, we obtain
	\begin{align}\label{phi0-0}
		\varphi_{Ray}^{(0)}(0)=-2\alpha c\int_{0}^{+\infty}\frac{e^{-2\alpha Z}}{(U(Z)-c)^2}dZ=\sum_{i=1}^{5}I_i(0),
	\end{align}
	where
	\begin{align}\begin{aligned}
			I_1(0)=&-2\alpha c\int_{Z_0}^{+\infty}\frac{e^{-2\alpha Z}dZ}{(U(Z)-c)^2},\ \  I_{2}(0) = ~\frac{2\alpha ce^{-2\alpha Z_0}}{U'(Z_c)}\left(\frac{\gamma}{U(Z_0)-c}+\frac{U''(Z_c)}{U'(Z_c)^2}\log(U(Z_0)-c)\right), \\
			I_{3}(0) =& ~\frac{2\alpha}{U'(Z_c)}\left(\gamma-\frac{U''(Z_c)}{U'(Z_c)^2}c\log(-c)\right), \ \
			I_{4}(0) =   ~\frac{2\alpha c}{U'(Z_c)}\int_{0}^{Z_0}\frac{e^{-2\alpha Z}g(Z)}{(U(Z)-c)^2}dZ,\\
			I_{5}(0) =&   ~\frac{4\alpha^2 c}{U'(Z_c)}\int_{0}^{Z_0}e^{-2\alpha Z}\left(\frac{\gamma}{U(Z)-c}+\frac{ U''(Z_c)}{U'(Z_c)^2}\log(U(Z)-c)\right)dZ.\nonumber
	\end{aligned}\end{align}

	We now estimate $I_1(0)$ to $I_5(0)$ term by term. Note that $I_1(0)$ encodes the boundary information received from the far field, while $I_3(0)$ represents the boundary contribution arising from the localized term. It turns out that only these two terms contribute at leading order in the expansions \eqref{Re1} and \eqref{Im1}.
	
	First, we estimate $I_1(0)$. From \eqref{kernel-large} we obtain
	\begin{align*}
			\int_{Z_0}^{+\infty}\frac{e^{-2\alpha Z'}dZ'}{(U(Z')-c)^2}
			&=\frac{e^{-2\alpha Z_0}}{2\alpha(U_\infty-c)^2}+A(Z_0)=\frac{e^{-2\alpha Z_0}}{2\alpha (U_\infty-c_r)^2}
			+\mathcal{O}(1)\frac{c_i}{\alpha}+A(Z_0),
		\end{align*}
		which implies
		\begin{align*}
			I_1(0) =& -\frac{ce^{-2\alpha Z_0}}{(U_\infty-c_r)^2}
			-2\alpha cA(Z_0)+\mathcal{O}(|c|c_i).
		\end{align*}
		Recall $A(Z)$ in \eqref{kernel-large}. A straightforward computation yields
		\begin{align*}
			|\text{Re}\big(A(Z_0)\big)|\lesssim 1, \text{ and } |\text{Im}\big(A(Z_0)\big)|\lesssim c_i.
		\end{align*}
		Consequently, by virtue of $c_i\ll c_r$ we deduce that
		\begin{align}\label{I1-0-part}\begin{aligned}
				\text{Re}\left( I_1(0)\right) & =-\frac{c_re^{-2\alpha Z_0}}{ (U_\infty-c_r)^2}+\mathcal{O}(1)(\alpha+c_i)|c|=-\frac{c_r}{ U_\infty^2}+\mathcal{O}(1)(\alpha+c_r)c_r,\\
				\text{Im}\left( I_1(0)\right) & =-\frac{c_ie^{-2\alpha Z_0}}{ (U_\infty-c_r)^2}+\mathcal{O}(1)(\alpha+|c|)c_i=-\frac{c_i}{ U_\infty^2}+o(1)(\alpha+c_r)c_r.
		\end{aligned}    \end{align}
		
		Next we consider $I_{2}(0)$. Note that
		\begin{align*}
			\left|\frac{\gamma}{U(Z_0)-c}\right|+\left|\frac{U''(Z_c)}{U'(Z_c)^2}\log(U(Z_0)-c)\right|&=\mathcal{O}(1), \\
			\text{Im}\left(\frac{\gamma}{U(Z_0)-c}+\frac{U''(Z_c)}{U'(Z_c)^2}\log(U(Z_0)-c)\right) &=\mathcal{O}(1)\big(c_i+\arg(U(Z_0)-c)\big)=\mathcal{O}(c_i),
		\end{align*}
		due to $\arg(U(Z_0)-c)\sim -c_i$.  Thus, by taking real and imaginary parts of $I_2(0)$ respectively, we obtain
		\begin{align}\label{I2-0-part}\begin{aligned}
				\text{Re}\big(I_{2}(0)\big)&=\mathcal{O}(1)\alpha |c|=\mathcal{O}(\alpha c_r),\\
				\text{Im}(I_{2}(0))&=\mathcal{O}(1)\alpha c_i\Big(1+ c_r\Big)
				=o(\alpha c_r).
		\end{aligned}\end{align}

	    For $I_3(0),$ by direct calculation and recalling $\gamma$ as defined in \eqref{gamma}, we obtain 
		\begin{align}\label{I3-0-part}\begin{aligned}
				\text{Re}\big(I_{3}(0)\big)= &~ \frac{2\alpha}{U'(Z_c)}\left(1-\frac{U''(Z_c)}{U'(Z_c)^2}\Big(c_r\log|c|-c_i\arg(-c)\Big)\right) \\
				= &~ \frac{2\alpha}{U'(Z_c)}\left(1-\frac{U''(Z_c)}{U'(Z_c)^2}c_r\log|c|\right)
				+\mathcal{O}(\alpha c_i),\\
				\text{Im}\big(I_{3}(0)\big)= & \frac{2\alpha}{U'(Z_c)}\left(-\frac{c_iU''(Z_c)}{U'(Z_c)^2}
				-\frac{U''(Z_c)}{U'(Z_c)^2}\big(c_r\arg(-c)+c_i\log|c|\big)\right) \\
				= & -\frac{2U''(Z_c)}{U'(Z_c)^3}\alpha\big(c_r\arg(-c)+c_i\log|c|\big)
				+\mathcal{O}(\alpha c_i).
		\end{aligned}\end{align}

		For $I_4(0)$, we recall the definition of $g(Z)$ in  \eqref{def-g}. Then, it holds that
		\begin{align}\label{g-part}
			\text{Re}\left(\frac{g(Z)}{(U(Z)-c)^2}\right)=\mathcal{O}(1),\quad  \text{Im}\left(\frac{g(Z)}{(U(Z)-c)^2}\right)
			=\mathcal{O}(1)\frac{c_i}{|U(Z)-c|},
		\end{align}
		which implies
		\begin{align*}
			&\text{Re}\left(\int_{0}^{Z_0}\frac{e^{-2\alpha Z'}g(Z)}{(U(Z)-c)^2}dZ'\right)  =\mathcal{O}(1),\\
			&\text{Im}\left(\int_{0}^{Z_0}\frac{e^{-2\alpha Z'}g(Z)}{(U(Z)-c)^2}dZ'\right)  =\mathcal{O}(1)c_i\int_{0}^{Z_0}\frac{dZ'}{|U(Z)-c|}=\mathcal{O}(1)c_i|\log c_i|.
		\end{align*}
		Substituting these bounds into $I_4(0)$, we deduce that
		\begin{align}\label{I4-0-part}
			\text{Re}\big(I_4(0)\big)  =\mathcal{O}(1)\alpha |c|,\quad
			\text{Im}\big(I_4(0)\big) =\mathcal{O}(1)\alpha\big(c_r c_i|\log c_i|+c_i\big)
			=o(\alpha c_r).
		\end{align}
		
		Finally, we can bound $I_5(0)$ as
		\begin{align}\label{I5-0-part}
			|I_5(0)|\leq C\alpha^2 |c\log c_i|=o(\alpha c_r),
		\end{align}
		where we have used $|\alpha \log c_i|\ll 1$. By substituting \eqref{I1-0-part}, \eqref{I2-0-part}, \eqref{I3-0-part}, \eqref{I4-0-part} and \eqref{I5-0-part} into \eqref{phi0-0}, and then taking the real and imaginary parts respectively, we obtain
		\begin{align}\begin{aligned}
				\text{Re}\left(\varphi_{Ray}^{(0)}(0)\right)= & -\frac{c_r}{U_\infty^2}+ \frac{2\alpha}{U'(Z_c)}\left(1-\frac{U''(Z_c)}{U'(Z_c)^2}c_r\log|c|\right) +\mathcal{O}\big((\alpha+c_r)c_r\big), \\
				\text{Im}\left(\varphi_{Ray}^{(0)}(0)\right)= &  -\frac{c_i}{U_\infty^2}- \frac{2\alpha U''(Z_c)}{U'(Z_c)^3}\Big(c_r\arg(-c)+c_i\log|c|\Big)
				+o\big((\alpha+c_r)c_r\big).\nonumber
		\end{aligned}\end{align}
		Thus, the proof of \eqref{Re1} and \eqref{Im1} is complete.\\

		\underline{\bf{ Estimates of $\partial_Z\varphi_{Ray}^{(0)}(0)$}}. Evaluating \eqref{phi0z} at $Z=0$, we obtain
		\begin{align}
				\partial_Z\varphi_{Ray}^{(0)}(0)
				= 2\alpha c^{-1}+\left(\alpha-U'(0) c^{-1}\right)\varphi_{Ray}^{(0)}(0).\label{Z1}
			\end{align}
			Substituting \eqref{Re1} and \eqref{Im1} into \eqref{Z1} yields
			\begin{align*}
					\text{Re}(\partial_Z\varphi_{Ray}^{(0)}(0)) =~& \frac{2\alpha c_r}{|c|^2}
					+\left(\alpha-\frac{U'(0)c_r}{|c|^2}\right)\text{Re}\big(\varphi_{Ray}^{(0)}(0)\big)
					-\frac{U'(0)c_i}{|c|^2}\text{Im}\big(\varphi_{Ray}^{(0)}(0)\big)\\
					=&~\frac{2\alpha c_r}{|c|^2}\left(1-\frac{ U'(0)}{U'(Z_c)}\right)
					+
					\frac{U'(0)c_r}{|c|^2}\left(\frac{c_r}{U_\infty^2}
					+\frac{2U''(Z_c)}{U'(Z_c)^3}\alpha c_r\log|c|\right)\\
					&+\frac{U'(0)c_i}{|c|^2}\left(\frac{c_i}{U_\infty^2}
					+\frac{2U''(Z_c)}{U'(Z_c)^3}\alpha \big(c_r\arg(-c)+c_i\log|c|\big)\right)+\mathcal{O}(\alpha+c_r)\\
					=&~\frac{U'(0)}{U_\infty^2}+\frac{2U''(Z_c)}{U'(Z_c)^2}\alpha\log|c|
					+\mathcal{O}(\alpha+c_r),
				\end{align*}
				and
				\begin{align*}
						\text{Im}(\partial_Z\varphi_{Ray}^{(0)}(0)) =& -\frac{2\alpha c_i}{|c|^2}
						+\frac{U'(0)c_i}{|c|^2}\text{Re}\big(\varphi_{Ray}^{(0)}(0)\big)
						+\left(\alpha-\frac{U'(0)c_r}{|c|^2}\right)\text{Im}\big(\varphi_{Ray}^{(0)}(0)\big)\\
						=&~\frac{2\alpha c_i}{|c|^2}\left(-1+\frac{ U'(0)}{U'(Z_c)}\right)
						-\frac{U'(0)c_i}{|c|^2}\left(\frac{c_r}{U_\infty^2}
						+\frac{2U''(Z_c)}{U'(Z_c)^3}\alpha c_r\log|c|\right)\\
						&+
						\frac{U'(0)c_r}{|c|^2}\left(\frac{c_i}{U_\infty^2}
						+\frac{2U''(Z_c)}{U'(Z_c)^3}\alpha \big(c_r\arg(-c)+c_i\log|c|\big)\right)+o(\alpha+c_r)\\
						=&~\frac{2U''(Z_c)}{U'(Z_c)^2}\alpha\arg(-c)
						+o(\alpha+c_r),
					\end{align*}
					where we have used the facts $c_i\ll\min\{\alpha, c_r\}$ and 
					\begin{align}
						U'(Z_c)-U'(0)\sim Z_c\sim c_r.\nonumber
					\end{align}
					 Thus, we achieve \eqref{Re2} and \eqref{Im2}, and  complete the proof of Lemma \ref{lem:ray-main-asy}.
					
				\end{proof}
				
				\subsection{Remainder estimate}

				In this part, we construct the solution $\varphi_{Ray}$ to the homogeneous Rayleigh equation \eqref{eq:Ray-homo} around the approximation $\varphi_{Ray}^{(0)}$. Recall that
				\begin{align*}
					\mathrm{Ray}[\varphi_{Ray}^{(0)}]=2\alpha U'\varphi_{Ray}^{(0)}.
				\end{align*}
		We then solve for $\varphi_{Ray}$ in the form 
				\begin{align*}
					\varphi_{Ray}(Z)=\varphi_{Ray}^{(0)}(Z)+\varphi_R(Z),\ \ \text{ where } \ Ray[\varphi_R]=-2\alpha U'\varphi_{Ray}^{(0)}.
				\end{align*}
			We furthre decompose
				\begin{align}\label{def-phi-R}
					\varphi_R=\varphi_1+\varphi_2.
				\end{align} 
			where
				\begin{align}\label{def-phi1}
					\varphi_1(Z)\triangleq-2\alpha(U(Z)-c)\int_Z^{+\infty} \frac{\int_{Z'}^{+\infty}U'(Z'')\varphi_{Ray}^{(0)}(Z'')dZ''}{(U(Z')-c)^{2}}dZ'.
				\end{align}
			It is straightforward to check that $\mathrm{Ray}[\varphi_1]=-2\alpha U'\varphi_{Ray}^{(0)}-\alpha^2(U-c)\varphi_1.$ Hence, $\varphi_{Ray}^{(0)}+\varphi_{1}$ is the approximation solution to the Rayleigh equation at $\mathcal{O}(\alpha)$. 
		The remainder $\varphi_2$ then solves
				\begin{align}
					\mathrm{Ray}[\varphi_2]=\alpha^2(U-c)\varphi_1.\nonumber
				\end{align}
	 Applying \eqref{eq-Ray-nonhom-Y} to $\varphi_1$ and $\varphi_2$, and using the bound \eqref{eq-lem-ray-sub-profile} for $\varphi^{(0)}_{Ray}$, we obtain
				
				\begin{lemma}\label{lem:ray-R}
				$\varphi_1$ and $\varphi_2$ satisfy the following bounds
				\begin{align}
					\|\varphi_1\|_{\mathcal{Y}_{\eta_0}}&\leq C\alpha|\log c_i|,\label{vphi1}\\
					\|\varphi_2\|_{\mathcal{Y}_{\eta_0}}&\leq C\alpha^3|\log c_i|^2.\label{vphi2}
				\end{align}
				\end{lemma}
			
	The bound of $\partial_c\varphi_R$ is given in the following lemma.
	\begin{lemma}\label{lmphir}	
		The remainder $\varphi_{R}$ satisfies
		\begin{align}
			\|\partial_c\varphi_R\|_{\Tilde{\mathcal{Y}}_{\eta_0}}\leq C\alpha|\log c_i|.\label{phir}
		\end{align}
	\end{lemma}
	\begin{proof}
		Note that $\partial_c\varphi_R$ satisfies the following equation
		\begin{align}\nonumber
			\mathrm{Ray}[\partial_c\varphi_R]=-2\alpha U'\partial_c\varphi_{Ray}^{(0)}+(\partial_Z^2-\alpha^2)\varphi_R.	
		\end{align}
		Applying the Rayleigh estimate \eqref{eq-Ray-nonhom-tilde-Y} to $\partial_c\varphi_R$, we deduce that
		\begin{align}
			\|\partial_c\varphi_R\|_{\Tilde{\mathcal{Y}}_{\eta_0}}&\leq C|\log c_i|\left(\|(\partial_Z^2-\alpha^2)\varphi_R\|_{L^\infty_{\eta_0}}+\alpha\|\partial_c\varphi_{Ray}^{(0)}\|_{L^\infty}	\right)\nonumber\\
			&\leq C\alpha |\log c_i|,\nonumber
		\end{align}
		where we have used bounds \eqref{phic}, \eqref{vphi1} and \eqref{vphi2} in the last step. Thus, the proof of Lemma \ref{lmphir} is complete.
	\end{proof}
	In the following lemma, we derive the asymptotic values of $\varphi_{R}(0)$, which are essential for the analysis of the dispersion relation.
				
				\begin{lemma}\label{lem:ray-R-asy}
					Let $(\alpha,c)\in \mathbb{H}_1$. If $c_i\ll \min\{\alpha, c_r\}$, then
					\begin{align}
						\mathrm{Re}\big(\varphi_R(0)\big) &= -\frac{\alpha}{U'(Z_c)}+\frac{U''(Z_c)}{U'(Z_c)^3} \alpha c_r\log|c|+\mathcal{O}\big(\alpha (\alpha+c_r)\big),\label{r1}\\
						\mathrm{Im}\left(\varphi_{R}(0)\right)&=\frac{ U''(Z_c)}{U'(Z_c)^3}
						\alpha\big(c_r\arg(-c)+c_i\log|c|\big)+o(\alpha (\alpha+c_r)),\label{r2}\\
						\mathrm{Re}\big(\partial_Z\varphi_R(0)\big)&=-\frac{U''(Z_c)}{U'(Z_c)^2}\alpha \log|c|+\mathcal{O}(\alpha+c_r),\label{r3}\\
						\mathrm{Im}\left(\partial_Z\varphi_{R}(0) \right)&=-\frac{ U''(Z_c)}{U'(Z_c)^2}\alpha\arg(-c) +o(\alpha).\label{r4}
					\end{align}
				\end{lemma}
				
				\begin{proof}
						Evaluating \eqref{def-phi-R} at $Z=0$, we have $\varphi_R(0)=\varphi_1(0)+\varphi_2(0)$. By \eqref{vphi2}, we obtain 
					\begin{align}
						|\partial_Z\varphi_2(0)|+|\varphi_2(0)|\lesssim \alpha^3|\log c_i|^2\ll \alpha^2,\label{phi2}
					\end{align}
					and it remains to study $\varphi_1(0)$. For this, we denote
						\begin{align}
							G^{(0)}(Z)\triangleq \int_{Z}^{+\infty}U'(Z')\varphi_{Ray}^{(0)}(Z')dZ'.\nonumber
						\end{align} 
		Then applying the decomposition \eqref{d} to \eqref{def-phi1}, we have	
						\begin{align}\label{phi1}\begin{aligned}
								\varphi_1(Z)=&-2\alpha(U(Z)-c)\int_{Z}^\infty\frac{G^{(0)}(Z')}{(U(Z')-c)^2}d Z'
								\\
								= & \frac{2\alpha (U(Z)-c)}{U'(Z_c)}\left(\int_{Z}^{+\infty}\frac{g(Z')G^{(0)}(Z')}{(U(Z')-c)^2}dZ'
								+\frac{U''(Z_c)}{U'(Z_c)^2}\int_{Z}^{+\infty}\frac{U'(Z')G^{(0)}(Z')}{U(Z')-c}dZ'\right.\\
								&\qquad\qquad\qquad\left.
								-\gamma\int_{Z}^{+\infty}\frac{U'(Z')G^{(0)}(Z')}{(U(Z')-c)^2}dZ'\right)\\
								\triangleq &\frac{2\alpha (U(Z)-c)}{U'(Z_c)}\Big(J_1(Z)+J_2(Z)+J_3(Z)\Big).
						\end{aligned}\end{align}
				Integrating by parts yields
						\begin{align}\label{j2}
							J_2(Z)= & \frac{U''(Z_c)}{U'(Z_c)^2}\left(-\log(U(Z)-c)
							G^{(0)}(Z)
							+\int_{Z}^{+\infty}\log(U(Z')-c)U'(Z')\varphi^{(0)}_{Ray}(Z') dZ'\right). \end{align}
				For $J_3$, by using the explicit formula \eqref{eq:ray-main-def} of $\varphi^{(0)}_{Ray}(Z)$ and integrating by parts,  we obtain
						\begin{align}
							J_3(Z)= & -\gamma\frac{G^{(0)}(Z)}{U(Z)-c}     +\gamma\int_{Z}^{+\infty}\frac{U'(Z')\varphi_{Ray}^{(0)}(Z')}{U(Z')-c}dZ'\nonumber\\
							= & -\gamma\frac{G^{(0)}(Z)}{U(Z)-c}     +2\alpha\gamma\int_{Z}^{+\infty}e^{\alpha Z'}\left(\int_{Z'}^{+\infty}\frac{e^{-2\alpha Z''}dZ''}{(U(Z'')-c)^2}\right)d(U(Z')-c)\nonumber\\
							= &-\gamma\frac{G^{(0)}(Z)}{U(Z)-c}-\gamma\varphi_{Ray}^{(0)}(Z)
							+\alpha\gamma\int_{Z}^{+\infty}\left(\frac{2e^{-\alpha Z'}}{U(Z')-c}-\varphi_{Ray}^{(0)}(Z')\right)dZ'.\label{j3}
						\end{align}
					Substituting \eqref{j2} and \eqref{j3} into \eqref{phi1}, we obtain the following decomposition for $\varphi_1$:
						\begin{align}\label{phi1-new}\begin{aligned}
								\varphi_1(Z)= &\sum_{i=1}^{4}K_i(Z), 
						\end{aligned}\end{align} 
						where
						\begin{align}\begin{aligned}
								K_1(Z)= &\frac{2\alpha U''(Z_c)}{U'(Z_c)^3}(U(Z)-c)\int_{Z}^{+\infty}
								\log(U(Z')-c)
								U'(Z')\varphi_{Ray}^{(0)}(Z')dZ',\\
								K_2(Z)= &\frac{2\alpha^2\gamma}{U'(Z_c)}(U(Z)-c)\int_{Z}^{+\infty}\left(\frac{2e^{-\alpha Z'}}{U(Z')-c}-\varphi_{Ray}^{(0)}(Z')\right)dZ',\\
								K_3(Z)=&\frac{2\alpha (U(Z)-c)}{U'(Z_c)}\int_{Z}^{+\infty}\frac{g(Z')G^{(0)}(Z')}{(U(Z')-c)^2}dZ',\\
								K_4(Z) =& -\frac{2\alpha }{U'(Z_c)}G^{(0)}(Z)\left(\gamma+\frac{U''(Z_c)}{U'(Z_c)^2}(U(Z)-c)\log(U(Z)-c)\right)\\
								&-\frac{2\alpha\gamma}{U'(Z_c)}(U(Z)-c)\varphi_{Ray}^{(0)}(Z).\nonumber
						\end{aligned}\end{align}
Evaluating \eqref{phi1-new} at $Z=0$, we obtain
			\begin{align}\label{est-phi1-0}\begin{aligned}
					\varphi_1(0)= & \sum_{i=1}^{4}K_i(0),
						\end{aligned}\end{align}
		where
			\begin{align}\begin{aligned}
					K_1(0)= &-\frac{2\alpha c U''(Z_c)}{U'(Z_c)^3}\int_{0}^{+\infty}
					\log(U(Z)-c) U'(Z)\varphi_{Ray}^{(0)}(Z)dZ,\\
					K_2(0)= &-\frac{2\alpha^2\gamma c}{U'(Z_c)}\int_{0}^{+\infty}\left(\frac{2e^{-\alpha Z}}{U(Z)-c}-\varphi_{Ray}^{(0)}(Z)\right)dZ,\\
					K_3(0)=&-\frac{2\alpha c}{U'(Z_c)}\int_{0}^{+\infty}\frac{g(Z)G^{(0)}(Z)}{(U(Z)-c)^2}dZ,\\
					K_4(0) =& \frac{2\alpha G^{(0)}(0)}{U'(Z_c)}\left(-\gamma+\frac{U''(Z_c)}{U'(Z_c)^2}c\log(-c)\right)
					+\frac{2\alpha\gamma c}{U'(Z_c)}\varphi_{Ray}^{(0)}(0).\nonumber
			\end{aligned}\end{align}
		We now estimate $K_1(0)$ to $K_4(0)$ term by term. It turns out that, among these, only the localized term $K_4(0)$ contributes at leading order in the expansions \eqref{r1} and \eqref{r2}.
		
	 First, we estimate $K_1(0)$.  From the asymptotic formula \eqref{p1} and \eqref{p2} of $\varphi^{(0)}_{Ray}(Z)$, we obtain 
			\begin{align}\begin{aligned}
					\varphi_{Ray}^{(0)}(Z)&=\frac{e^{-\alpha Z}(U(Z)-c)}{(U_\infty-c)^2}+\mathcal{O}(1)\alpha|\log c_i|\\
					&=\frac{e^{-\alpha Z}(U(Z)-c_r)}{(U_\infty-c_r)^2}+\mathcal{O}(1)\big(c_i+\alpha|\log c_i|\big).\nonumber
			\end{aligned}\end{align}
		Then it holds that
			\begin{align*}
				&\int_{0}^{+\infty}\log(U(Z)-c) U'(Z)\varphi_{Ray}^{(0)}(Z)dZ\\
				&= \int_{0}^{+\infty}\log(U(Z)-c) U'(Z)\frac{e^{-\alpha Z}(U(Z)-c_r)}{(U_\infty-c_r)^2}dZ+\mathcal{O}(1)\big(c_i+\alpha |\log c_i|\big)\\
				&=\mathcal{O}(1).
			\end{align*}
			Moreover, note that
			\begin{align*}
				&\text{Im}\Big((U(Z)-c_r)\log(U(Z)-c)\Big)= (U(Z)-c_r)\arg(U(Z)-c)\\
				&=\begin{cases}
					(U(Z)-c_r)\arctan\left(\frac{-c_i}{U(Z)-c_r}\right)=\mathcal{O}(c_i), & \mbox{if } Z\geq Z_c, \\
					(U(Z)-c_r)\left[\arctan\left(\frac{-c_i}{U(Z)-c_r}\right)-\pi\right]=\mathcal{O}(c_i+c_r), & \mbox{if~} 0\leq Z\leq Z_c,
				\end{cases} 
			\end{align*}
			where we have used the boundedness of function $\displaystyle \frac{\arctan x}{x}$. Thus, we obtain
			\begin{align*}
				&\text{Im}\left(\int_{0}^{+\infty}\log(U(Z)-c) U'(Z)\frac{e^{-\alpha Z}(U(Z)-c_r)}{(U_\infty-c_r)^2}dZ\right)\\
				& =\left(\int_{Z_c}^{+\infty}+\int_{0}^{Z_c}\right)\text{Im}\Big((U(Z)-c_r)\log(U(Z)-c) \Big)\frac{e^{-\alpha Z}U'(Z)}{(U_\infty-c_r)^2}dZ\\
				&=\mathcal{O}(c_i)+\mathcal{O}\big((c_i+c_r)Z_c\big)=\mathcal{O}(c_i+c_r^2),
			\end{align*}
			where we have used $Z_c\sim c_r$ in the last step. Consequently, it follows that
			\begin{align}\begin{aligned}
					\text{Re}\left(\int_{0}^{+\infty}\log(U(Z)-c) U'(Z)\varphi_{Ray}^{(0)}(Z)dZ\right) & =\mathcal{O}(1)\big(1+\alpha|\log c_i|+c_i\big)=\mathcal{O}(1), \\
					\text{Im}\left(\int_{0}^{+\infty}\log(U(Z)-c) U'(Z)\varphi_{Ray}^{(0)}(Z)dZ\right) & =\mathcal{O}(1)\big(\alpha|\log c_i|+c_i+c_r^2\big).\nonumber
			\end{aligned}\end{align}
Substituting these two equalities into $K_1(0)$ yields
			\begin{align}\label{K1-0-im}\begin{aligned}
					\text{Re}\big(K_1(0)\big) =& -\frac{2\alpha c_r U''(Z_c)}{U'(Z_c)^3} \text{Re}\left(\int_{0}^{+\infty}\log(U(Z)-c) U'(Z)\varphi_{Ray}^{(0)}(Z)dZ\right)\\
					&+\frac{2\alpha c_i U''(Z_c)}{U'(Z_c)^3} \text{Im}\left(\int_{0}^{+\infty}\log(U(Z)-c) U'(Z)\varphi_{Ray}^{(0)}(Z)dZ\right)\\
					= &~ \mathcal{O}(\alpha c_r),\\
					\text{Im}\big(K_1(0)\big) =& -\frac{2\alpha c_r U''(Z_c)}{U'(Z_c)^3} \text{Im}\left(\int_{0}^{+\infty}\log(U(Z)-c) U'(Z)\varphi_{Ray}^{(0)}(Z)dZ\right)\\
					&-\frac{2\alpha c_iU''(Z_c)}{U'(Z_c)^3} \text{Re}\left(\int_{0}^{+\infty}\log(U(Z)-c) U'(Z)\varphi_{Ray}^{(0)}(Z)dZ\right)\\
					= &~ \mathcal{O}(1)\alpha c_r\big(\alpha|\log c_i|+c_i+c_r^2\big)+\mathcal{O}(1)\alpha c_i=~o(\alpha c_r).
			\end{aligned}\end{align}
		
For $K_2(0)$, from the asymptotic formula \eqref{p1} and \eqref{p2} of $\varphi^{(0)}_{Ray}$, we obtain
			\begin{align*}
				& \int_{0}^{+\infty}\left(\frac{2e^{-\alpha Z}}{U(Z)-c}-\varphi_{Ray}^{(0)}(Z)\right)dZ\\
				&=\int_{0}^{+\infty}\left(\frac{2e^{-\alpha Z}}{U(Z)-c}-\frac{e^{-\alpha Z}(U(Z)-c)}{(U_\infty-c)^2}\right)dZ+\mathcal{O}(1)\alpha|\log c_i|\\
				&=\int_{0}^{+\infty}\left(\frac{2e^{-\alpha Z}}{U_\infty-c}-\frac{e^{-\alpha Z}}{U_\infty-c}\right)dZ\\
				&\quad+\mathcal{O}(1)\int_{0}^{+\infty}|U_\infty-U(Z)|\left(\frac{1}{|U(Z)-c|}+1\right)dZ+\mathcal{O}(1)\alpha|\log c_i|\\
				&=\frac{1}{\alpha(U_\infty-c)}+\mathcal{O}(1)|\log c_i|.
			\end{align*}
		Substituting it into $K_{2}(0)$, we deduce that
			\begin{align*}
				K_2(0)= & -\frac{2\alpha\gamma c}{U'(Z_c)(U_\infty-c)}
				+\mathcal{O}(1)\alpha^2|c\log c_i|.
			\end{align*}
		Taking real and imaginary parts of the above equality respectively, we get
			\begin{align}\label{K2-0-im}\begin{aligned}
					\text{Re}\big(K_2(0)\big) =& \mathcal{O}(\alpha c_r),\quad
					\text{Im}\big(K_2(0)\big) =& \mathcal{O}(1)\big(\alpha c_i+\alpha^2|c\log c_i|\big)=o(\alpha c_r).
			\end{aligned}\end{align}

The estimate of $K_3(0)$ is more subtle. We need to establish a more accurate explicit expression of $G^{(0)}(Z)$. Substituting \eqref{eq:ray-main-def} into $G^{(0)}(Z)$ and integrating by parts, we obtain
			\begin{align}\label{G-new}\begin{aligned}
					G^{(0)}(Z) =& \alpha\int_{Z}^{+\infty}e^{\alpha Z'}\left(\int_{Z'}^{+\infty}\frac{e^{-2\alpha Z''}d Z''}{(U(Z'')-c)^2}\right)d(U(Z')-c)^2\\
					=& e^{-\alpha Z}-\frac{1}{2}(U(Z)-c)\varphi_{Ray}^{(0)}(Z)
					-\frac{\alpha}{2}\int_{Z}^{+\infty}(U(Z')-c)\varphi_{Ray}^{(0)}(Z')dZ'.
			\end{aligned}\end{align}
			Thanks to \eqref{p1} and \eqref{p2}, we rewrite
			\begin{align*}
				(U(Z)-c)\varphi_{Ray}^{(0)}(Z)= &~ e^{-\alpha Z}\frac{(U(Z)-c)^2}{(U_\infty-c)^2}+\mathcal{O}(1)\alpha|\log c_i||U(Z)-c|e^{-\eta_0 Z}.
			\end{align*}
	Introduce $h(Z)\triangleq \frac{(U(Z)-c)^2}{(U_\infty-c)^2}-1$. It is clear that $h(Z) \in L^\infty_{\eta_0}$ and 
			\begin{align}\label{def-h}
			h(Z)= \frac{(U(Z)-c_r)^2}{(U_\infty-c_r)^2}-1+\mathcal{O}(1)c_ie^{-\eta_0 Z}.
			\end{align}
	Then from \eqref{G-new}, we deduce that
			\begin{align}\label{G-exp}\begin{aligned}
					G^{(0)}(Z)= &~e^{-\alpha Z}-\frac{e^{-\alpha Z}}{2}\frac{(U(Z)-c)^2}{(U_\infty-c)^2}-\frac{\alpha}{2}\int_{Z}^{+\infty}\frac{(U(Z')-c)^2}{(U_\infty-c)^2}e^{-\alpha Z'}dZ'\\
					&+\mathcal{O}(1)\big(\alpha|\log c_i||U(Z)-c|e^{-\eta_0 Z}+\alpha^2|\log c_i|e^{-\eta_0 Z}\big)\\
					=&-\frac{ e^{-\alpha Z}h(Z)}{2}-\frac{\alpha}{2}\int_{Z}^{+\infty}
					e^{-\alpha Z'}h(Z')dZ'+\mathcal{O}(1)\alpha|\log c_i|\big(|U(Z)-c|+\alpha\big)e^{-\eta_0 Z}.
			\end{aligned}\end{align}
	Substituting \eqref{def-h} into \eqref{G-exp}, then taking real and imaginary parts respectively, we deduce that
			\begin{align}\label{G-part}\begin{aligned}
					\text{Re}\big(G^{(0)}(Z)\big)=&~\frac{e^{-\alpha Z}}{2}\left(1-\frac{(U(Z)-c_r)^2}{(U_\infty-c_r)^2}\right) +\mathcal{O}(1)\Big(c_i+\alpha+\alpha|\log c_i|\big(|U(Z)-c|+\alpha\big)\Big)e^{-\eta_0Z},\\
					=&~\mathcal{O}(1)e^{-\eta_0Z},\\
					\text{Im}\big( G^{(0)}(Z)\big)=&~\mathcal{O}(1)c_ie^{-\eta_0Z}
					+\mathcal{O}(1)\alpha|\log c_i|\big(|U(Z)-c|+\alpha\big)e^{-\eta_0 Z}.
			\end{aligned}\end{align}
			Thus, it along with \eqref{g-part} yields 
			\begin{align}\begin{aligned}
					\text{Re}\left(\int_{0}^{+\infty}\frac{g(Z)G^{(0)}(Z)}{(U(Z)-c)^2}dZ\right) =&~ \mathcal{O}(1)\int_{0}^{+\infty}e^{-\eta_0Z}\left(1+\frac{c_i}{|U(Z)-c|} \left(c_i+\alpha|\log c_i|\right)\right)dZ,\\
					=&~ \mathcal{O}(1),\\
					\text{Im}\left(\int_{0}^{+\infty}\frac{g(Z)G^{(0)}(Z)}{(U(Z)-c)^2}dZ\right) =&~ \mathcal{O}(1)\int_{0}^{+\infty}e^{-\eta_0Z}\left(\left(c_i+\alpha|\log c_i|\right)+\frac{c_i}{|U(Z)-c|} \right)dZ.\\
					=&~ \mathcal{O}(1)(\alpha+c_i)|\log c_i|.\nonumber
			\end{aligned}\end{align}
	Substituting these bounds into $K_3(0)$, we arrive at
			\begin{align}\label{K3-0-im}\begin{aligned}
					\text{Re}\big(K_3(0)\big) =& -\frac{2\alpha}{U'(Z_c)}\left(c_r\text{Re}\left(\int_{0}^{+\infty}\frac{g(Z)G^{(0)}(Z)}{(U(Z)-c)^2}dZ\right) -c_i\text{Im}\left(\int_{0}^{+\infty}\frac{g(Z)G^{(0)}(Z)}{(U(Z)-c)^2}dZ\right) \right)\\
					=&~\mathcal{O}(\alpha c_r),\\
					\text{Im}\big(K_3(0)\big) =& -\frac{2\alpha}{U'(Z_c)}\left(c_r\text{Im}\left(\int_{0}^{+\infty}\frac{g(Z)G^{(0)}(Z)}{(U(Z)-c)^2}dZ\right) +c_i\text{Re}\left(\int_{0}^{+\infty}\frac{g(Z)G^{(0)}(Z)}{(U(Z)-c)^2}dZ\right) \right)\\
					=&~\mathcal{O}(1)\alpha\Big(c_r(\alpha+c_i)|\log c_i|+c_i\Big)=o(\alpha c_r).
			\end{aligned}\end{align}

Finally, we evaluate $K_4(0)$. Evaluating \eqref{G-part} at $Z=0$, we obtain 
			\begin{align}\label{G-0-part}\begin{aligned}
					\text{Re}\big( G^{(0)}(0)\big)= &~ \frac{1}{2}\left(1-\frac{c_r^2}{(U_\infty-c_r)^2}\right) +\mathcal{O}(1)\Big(c_i+\alpha+\alpha|\log c_i|\big(|c|+\alpha\big)\Big)\\
					=&~\frac{1}{2}+\mathcal{O}(1)\big(c_r+\alpha\big),\\
					\text{Im}\big( G^{(0)}(0)\big)
					=&~\mathcal{O}(1)\big(c_i+\alpha|\log c_i|(|c|+\alpha)\big).
			\end{aligned}\end{align}
Substituting \eqref{G-0-part} and \eqref{phi0-0-exp} into $K_4(0)$, and recalling \eqref{gamma} the definition of $\gamma$, we deduce
			\begin{align}\label{K4-0-im}\begin{aligned}
					\text{Re}\big(K_4(0)\big) =& \frac{2\alpha}{U'(Z_c)}\left\{\text{Re}\big(G^{(0)}(0)\big) \text{Re}\left(-\gamma+\frac{U''(Z_c)}{U'(Z_c)^2}c\log(-c)\right)\right.\\
					&\qquad\quad \left.-\text{Im}\big(G^{(0)}(0)\big) \text{Im}\left(-\gamma+\frac{U''(Z_c)}{U'(Z_c)^2}c\log(-c)\right)\right\}+o(\alpha|c|)\\
					=&~-\frac{\alpha}{U'(Z_c)}+\frac{U''(Z_c)}{U'(Z_c)^3} \alpha c_r\log|c|+\mathcal{O}\big(\alpha (\alpha+c_r)\big),
			\end{aligned}
		\end{align}
	and
	\begin{align}\begin{aligned}\label{K4-0-re}
					\text{Im}\big(K_4(0)\big) =& \frac{2\alpha}{U'(Z_c)}\left\{\text{Re}\big(G^{(0)}(0)\big) \text{Im}\left(-\gamma+\frac{U''(Z_c)}{U'(Z_c)^2}c\log(-c)\right)\right.\\
					&\qquad\quad \left.+\text{Im}\big(G^{(0)}(0)\big) \text{Re}\left(-\gamma+\frac{U''(Z_c)}{U'(Z_c)^2}c\log(-c)\right)\right\}+o(\alpha|c|)\\
					=&~\frac{U''(Z_c)}{U'(Z_c)^3} \alpha\big(c_r\arg(-c)+c_i\log|c|\big)+o\big(\alpha (\alpha+c_r)\big).
			\end{aligned}\end{align}
		
	We now substitute \eqref{K1-0-im}, \eqref{K2-0-im}, \eqref{K3-0-im} \eqref{K4-0-im}, and \eqref{K4-0-re} into \eqref{est-phi1-0}, and obtain
			\begin{align}\label{phi1-0-im}\begin{aligned}
					\text{Re}\big(\varphi_1(0)\big) =& -\frac{\alpha}{U'(Z_c)}+\frac{U''(Z_c)}{U'(Z_c)^3} \alpha c_r\log|c|+\mathcal{O}\big(\alpha (\alpha+c_r)\big),\\
					\text{Im}\big(\varphi_1(0)\big) =&~ \frac{U''(Z_c)}{U'(Z_c)^3} \alpha\big(c_r\arg(-c)+c_i\log|c|\big)+o\big(\alpha (\alpha+c_r)\big).
			\end{aligned}\end{align}
		Thus, the bounds \eqref{r1} and \eqref{r2} follow from \eqref{phi2} and \eqref{phi1-0-im}.\\
		
	Finally, we establish the bound for $\partial_Z\varphi_1(0)$. Taking derivative in \eqref{def-phi1}, and evaluating the resultant equation at $Z=0$, we obtain
\begin{align}
	\partial_Z\varphi_1(0)=-2\alpha c^{-1}G^{(0)}(0)-U'(0)c^{-1}\varphi_1(0).\label{phid}
\end{align}
Then substituting \eqref{G-0-part} and \eqref{phi1-0-im} into \eqref{phid}, we deduce that
\begin{align}
	\begin{aligned}
		\text{Re}\big(\partial_Z\varphi_1(0)\big)&=-\frac{2\alpha}{|c|^2}\left(
		c_r\text{Re}\big(G^{(0)}(0)\big)+c_i\text{Im}\big(G^{(0)}(0)\big)
		\right)-\frac{U'(0)}{|c|^2}\left(
		c_r\text{Re}\big(\varphi_1(0)\big)+c_i\text{Im}\big(\varphi_1(0)\big)
		\right)\\
		&=\frac{\alpha c_r}{|c|^2}\left(-1+\frac{U'(0)}{U'(Z_c)}\right)-\frac{U'(0)U''(Z_c)\alpha c_r^2\log|c|}{U'(Z_c)^3|c|^2}+\mathcal{O}(\alpha+c_r)\\
		&=-\frac{U''(Z_c)}{U'(Z_c)^2}\alpha \log|c|+\mathcal{O}(\alpha+c_r),\\
		\text{Im}\big(\partial_Z\varphi_1(0)\big)&=-\frac{2\alpha}{|c|^2}\left(
		c_r\text{Im}\big(G^{(0)}(0)\big)-c_i\text{Re}\big(G^{(0)}(0)\big)
		\right)-\frac{U'(0)}{|c|^2}\left(
		c_r\text{Im}\big(\varphi_1(0)\big)-c_i\text{Re}\big(\varphi_1(0)\big)
		\right)\\
		&=\frac{\alpha c_r}{|c|^2}\left(1-\frac{U'(0)}{U'(Z_c)}\right)-\frac{U'(0)U''(Z_c)\alpha c_r^2\text{arg}(-c)}{U'(Z_c)^3|c|^2}+o(\alpha)\\
		&=-\frac{U''(Z_c)}{U'(Z_c)^2}\alpha \text{arg}(-c)+o(\alpha).\label{phid0}
	\end{aligned}
\end{align}
Thus, the bounds \eqref{r3} and \eqref{r4} follow from \eqref{phi2} and \eqref{phid0}. The proof of Lemma \ref{lem:ray-R-asy} is then complete.
	\end{proof}

\subsection{Inviscid instability}
We first present results concerning the homogeneous Rayleigh solution $\varphi_{Ray}$.
	\begin{proposition}\label{prop:ray-homo}
		Let $(\alpha,c)\in \mathbb{H}_1$. There exists a solution $\varphi_{Ray}\in {\mathcal{Y}}_\alpha$ to \eqref{eq:Ray-homo} satisfying
		\begin{align}
			\|\varphi_{Ray}-(U_\infty-c)^{-2}(U-c)e^{-\alpha Z}\|_{\mathcal{Y}_{\eta_0}}\leq C\alpha|\log c_i|^2.\label{bd}
		\end{align}
		If $c_i\ll \min\{\alpha,c_r\}$, then it holds that 
		\begin{align}
			& \mathrm{Re}(\varphi_{Ray}(0))=-\frac{c_r}{U_\infty^2}+\frac{\alpha}{U'(0)}+\mathcal{O}(\alpha|c||\log c|),\label{R1}\\ & \mathrm{Im}(\varphi_{Ray}(0))=-\frac{c_i}{U_\infty^2} -\frac{U''(0)}{U'(0)^3}\alpha (-\pi c_r+c_i\log|c|)+o(\alpha(\alpha+c_r)),\label{R2}\\ & \mathrm{Re}(\partial_Z\varphi_{Ray}(0))=\frac{U'(0)}{U_\infty^2}+\mathcal{O}(\alpha|\log c|),\label{R3}
			\\ & 
			 \mathrm{Im}(\partial_Z\varphi_{Ray}(0))=-\frac{ \pi U''(0)}{U'(0)^2}\alpha+o(\alpha+c_r).\label{R4}
		\end{align}
Moreover, $\varphi_{Ray}(0;c)$ is analytic in $c$, and satisfies
	\begin{align}\label{phicd}
		\partial_c\varphi_{Ray}(0;c)=\frac{-1}{U_{\infty}^2}+O(1)(\alpha|\log c_i|+|c|).
	\end{align}
	\end{proposition}
\begin{proof}
	Recall that $\varphi_{Ray}=\varphi_{Ray}^{(0)}+\varphi_1+\varphi_2$. Then the bound \eqref{bd} follows from \eqref{eq-lem-ray-sub-profile}, \eqref{vphi1} and \eqref{vphi2}. If $c_i\ll c_r$, we obtain $Z_c\sim c_r\sim |c|$, and
	$$\arg(-c)=-\pi+o(1).
	$$
	Consequently, the asymptotic expansions \eqref{R1}-\eqref{R4} follow from \eqref{Re1}-\eqref{Im2} in Lemma \ref{lem:ray-main-asy}, and \eqref{r1}-\eqref{r4} in Lemma \ref{lem:ray-R-asy}. Finally, the bound \eqref{phicd} follows from \eqref{dphic} and \eqref{phir}. The proof of Proposition \ref{prop:ray-homo} is complete.
\end{proof}
By using the asymptotic expansions of $\varphi_{Ray}(0;c)$ established in Proposition \ref{prop:ray-homo}, we obtain the following result.
\begin{theorem}[Inviscid instability]\label{iv}
	For any $\alpha\ll1$, there exists a complex number $c_{Ray}\in \mathbb{C}$ such that
	\begin{align}
		\varphi_{Ray}(0;c)|_{c=c_{Ray}}=0.\nonumber
	\end{align}
Moreover, $c_{Ray}$ admits the following asymptotic expansion
\begin{align}
	\mathrm{Re}(c_{Ray})=\frac{\alpha U_{\infty}^2}{U'(0)}+O(\alpha^2 |\log \alpha|),~~~\mathrm{Im}(c_{Ray})=\frac{\alpha^2U''(0)U_{\infty}^4\pi}{U'(0)^4}+o(\alpha^2).\label{eigen}
\end{align}
\end{theorem}
\begin{proof}
	We seek a root $c=c_r+\mathrm{i}c_i$ of the equation $\varphi_{Ray}(0;c)=0$ in the regime 
	\begin{align}
		c_r\sim \alpha,~~c_i\sim \alpha^2.\nonumber
	\end{align}
Set the approximate eigenvalue
\begin{align}
	c_{app}=\frac{\alpha U_{\infty}^2}{U'(0)}+\mathrm{i}\frac{\alpha^2U''(0)U_{\infty}^4\pi}{U'(0)^4}.\nonumber
\end{align}
From \eqref{R1} and \eqref{R2}, we obtain that
\begin{align}\label{phiray}
	\varphi_{Ray}(0;c)=\frac{-c+c_{app}}{U_{\infty}^2}+R(\alpha,c),
\end{align}
where $R(\alpha,c)$ is an analytic function in $c$, and satisfies
\begin{align}
	\mathrm{Re}\left(R(\alpha,c)\right)=O(\alpha^2|\log\alpha|),~~~\mathrm{Im}\left(R(\alpha,c)\right)=o(\alpha^2).\label{rd}
\end{align}
Moreover, by comparing \eqref{phiray} with \eqref{phicd}, we find that
\begin{align}
|\partial_c R(\alpha,c)|\leq C\alpha |\log c_i|\leq C\alpha |\log \alpha|.\label{rc}
\end{align}

Now we solve $c_{Ray}$ via the following iteration
$$c^{(j+1)}=c_{app}+U_{\infty}^2R(\alpha,c^{(j)}),~~c^{(0)}=c_{app}.
$$
From \eqref{rc}, we obtain
$$|c^{(j+1)}-c^{(j)}|=U_{\infty}^2|R(\alpha,c^{(j)})-R(\alpha,c^{(j-1)})|\leq C\alpha|\log \alpha||c^{(j)}-c^{(j-1)}|.
$$
Thus, for $0<\alpha\ll 1$, $\{c^{(j)}\}_{j\geq 0}$ is a Cauchy sequence. Set $c_{Ray}=\lim_{j\rightarrow\infty}c^{(j)}$, and it is straightforward to show that $c_{Ray}$ is a solution of the equation $\varphi_{Ray}(0;c)=0$. Moreover, substituting \eqref{rd} into \eqref{phiray}, we get the asymptotic expansion \eqref{eigen}. The proof of Theorem \ref{iv} is complete.
\end{proof}



\section{The Orr-Sommerfeld equation}\label{sec-OS}
In this section, we study the following Orr-Sommerfeld equation 
\begin{align}\label{eq:OSeq}
		-\varepsilon (\partial_Z^2-\alpha^2)^2\phi +(U-c)(\partial_Z^2-\alpha^2)\phi-U''\phi=0,
\end{align}
where $\varepsilon=\frac{\sqrt{\nu}}{i\alpha}.$ We construct a homogeneous Orr-Sommerfeld solution $\phi_s$ near the Rayleigh solution $\varphi_{Ray}$, and a viscous sublayer $\phi_f$ to match the full boundary conditions in \eqref{eq:OS-w}. These solutions are constructed with parameters $(\alpha,c)\in \mathbb{H}_2$ where
\begin{align*}
	\mathbb{H}_2&=\{(\alpha,c)\in \mathbb{R}_{+}\times\mathbb{C}\mid \alpha, c_r, c_i\sim O(1),~ c_i\ll c_r\sim \alpha\ll1,~~ \alpha|\log c_i| \ll 1\}.
\end{align*} 

\subsection{The Airy equation}\label{subsection-energy}

In this subsection, we solve the following Airy equation: 
\begin{align}\label{eq:Airy-eq-energy}
	\left\{\begin{aligned}
		&\mathrm{Airy}[\psi]=-\varepsilon(\partial_Z^2-\alpha^2)\psi+(U-c)\psi =F,\quad Z>0,\\ &\psi(0)=\psi(\infty)=0.
	\end{aligned}
	\right.	
\end{align}
\begin{lemma}\label{prop-airy-energy}
	Let $0<|\varepsilon|\ll 1$ and $(\alpha,c)\in \mathbb{H}_2$. For any $F\in L^2(\mathbb{R}_+)$, there exists a unique solution $\psi$ to \eqref{eq:Airy-eq-energy}, and it satisfies
	\begin{align}\label{est-L2-psi-energy}
			c_i\|\psi\|_{L^2}+|\varepsilon|^{\frac12} c_i^{\frac{1}{2}}\|(\partial_Z\psi,\alpha\psi)\|_{L^2}+|\varepsilon|c_i^{\frac12}\|(\partial_Z^2-\alpha^2)\psi\|_{L^2}\leq C\|F\|_{L^2}.
		\end{align}
\end{lemma}

\begin{proof}
	The existence is standard. We only focus on the a priori estimates \eqref{est-L2-psi-energy}. Taking the inner product of \eqref{eq:Airy-eq-energy} 
	with $\psi$ leads to
	\begin{align*}
		\left\langle F, {\psi} \right\rangle&=\left\langle -\varepsilon(\partial_Z^2-\alpha^2)\psi+(U-c)\psi, {\psi} \right\rangle\\ &=\varepsilon\|(\partial_Z\psi,\alpha\psi)\|^2_{L^2}+\int_0^{+\infty}(U-c_r)|\psi|^2dZ-\mathrm{i}c_i\|\psi\|^2_{L^2}\\ & =-\mathrm{i}|\varepsilon|\|(\partial_Z\psi,\alpha\psi)\|^2_{L^2}-\mathrm{i}c_i\|\psi\|^2_{L^2}+\int_0^{+\infty}(U-c_r)|\psi|^2dZ.
	\end{align*}
	The imaginary part of the above equality yields
	\begin{align*}
		|\varepsilon|\|(\partial_Z\psi,\alpha\psi)\|^2_{L^2}+c_i\|\psi\|^2_{L^2}\leq \left|\left\langle F, {\psi} \right\rangle\right|\leq C\|F\|_{L^2}\|\psi\|_{L^2},
	\end{align*}
	which implies that
	\begin{align}\label{est-psi-psiZalphapsi-energy}
		c_i\|\psi\|_{L^2}\leq C\|F\|_{L^2},\qquad |\varepsilon|^{\frac{1}{2}}c_i^{\frac12}\|(\partial_Z\psi,\alpha\psi)\|_{L^2}\leq C\|F\|_{L^2}.
	\end{align}
	Taking the $L^2$ norm on both sides of the first equation in \eqref{eq:Airy-eq-energy}, we obtain 
	\begin{align}\label{eq:F-inner-product}
		|\varepsilon|^2\|(\partial_Z^2-\alpha^2)\psi\|_{L^2}^2+\|(U-c)\psi\|_{L^2}^2-2\mathrm{Re}\left(\int_0^{+\infty}\varepsilon(\partial_Z^2-\alpha^2)\psi(U-\bar{c})\bar{\psi}dZ\right)=\|F\|^2_{L^2}.
	\end{align}
	By integrating by parts, we deduce that
	\begin{align*}
		\left|2\mathrm{Re}\left(\int_0^{+\infty}\varepsilon(\partial_Z^2-\alpha^2)\psi(U-\bar{c})\bar{\psi}dZ\right)\right|
		&
		{=2}|\varepsilon|\left|\mathrm{Im}\left(\int_0^{+\infty}(\partial_Z^2-\alpha^2)\psi(U-\bar{c})\bar{\psi}dZ\right)\right|\\ &\leq C|\varepsilon|\left(\|\partial_Z\psi\|_{L^2}\|\psi\|_{L^2}+\|\partial_Z\psi\|_{L^2}^2+\|\alpha\psi\|^2_{L^2}\right)\\ &\leq Cc_i^{-1}\|F\|_{L^2}^2,
	\end{align*}
	where we have used \eqref{est-psi-psiZalphapsi-energy} in the last inequality. Substituting it into \eqref{eq:F-inner-product} yields that
	\begin{align*}
		|\varepsilon|c_i^{\frac12}\|(\partial_Z^2-\alpha^2)\psi\|_{L^2}\leq C\|F\|_{L^2}.
	\end{align*}
Thus, the proof of Lemma \ref{prop-airy-energy} is complete.
	
\end{proof}

\subsection{Rayleigh-Airy itertion}
In this subsection, we solve the following non-homogeneous Orr-Sommerfeld equation 
\begin{align}\label{eq:non-OS}
    \mathrm{OS}[\phi]=-\varepsilon (\partial_Z^2-\alpha^2)^2\phi +(U-c)(\partial_Z^2-\alpha^2)\phi-U''\phi=F, \qquad \phi(\infty)=0,
\end{align}
with a general source $F$. The solution is constructed by the celebrated Rayleigh-Airy iteration, which has been introduced in \cite{GGN-DMJ, GMM-duke}. In our problem the convergence of this iteration must be justified in the regime $\alpha,c\sim O(1)$, and in the presence of inflection points within the profile. To this end, we employ the $L^\infty$ and $L^2$ framework to solve the Rayleigh and Airy equation respectively.

Recall the operators 
\begin{align*}
	\mathrm{OS}[~\cdot~]&= -\varepsilon (\partial_Z^2-\alpha^2)^2 +(U-c)(\partial_Z^2-\alpha^2)-U'',\\
	\mathrm{Ray}[~\cdot~]&=(U-c)(\partial_Z^2-\alpha^2)-U'',\\
	\mathrm{Airy}[~\cdot~]&=-\varepsilon (\partial_Z^2-\alpha^2)+(U-c).
\end{align*}
It is straightforward to check that 
\begin{align*}
	\mathrm{OS}[\phi]&= (\partial_Z^2-\alpha^2)\mathrm{Airy}[\phi]-2\partial_Z(\phi U')\\
	&=\mathrm{Ray}[\phi]-\varepsilon (\partial_Z^2-\alpha^2)^2\phi.
\end{align*}
We then solve \eqref{eq:non-OS} by the iteration, beginning with the Rayleigh solution $\varphi^{(0)}$, which satisfies
\begin{align}\label{eq:iter1-1}
    \mathrm{Ray}[\varphi^{(0)}]=F.
\end{align}
For $j\geq 1$, we inductively construct $\varphi^{(j)}$ and $\psi^{(j)}$  as follows:
\begin{align}\label{eq:iter2-1}
	\mathrm{Airy}[\psi^{(j)}] &=\varepsilon (\partial_Z^2-\alpha^2)\varphi^{(j-1)},
\end{align}
and
\begin{align} \label{eq:iter2-2} 
\mathrm{Ray}[\varphi^{(j)}]=2\partial_Z(U'\psi^{(j)}).
\end{align}
Then the $N$-th order approximate solution $\phi^{(N)}$ is defined by
\begin{align*}
    \phi^{(N)} \triangleq\varphi^{(N)}+\psi^{(N)}.
\end{align*}
Straightforward computation yields
\begin{align*}
    \mathrm{OS}\Big[\varphi^{(0)}+\sum_{j=1}^N \phi^{(j)}\Big]=F-\varepsilon(\partial_Z^2-\alpha^2)\varphi^{(N)}.
\end{align*}
Therefore, at this point, the series $\displaystyle\phi=\varphi^{(0)}+\sum_{j=1}^\infty \phi^{(N)}$ formally gives a solution to \eqref{eq:non-OS}. 

Now we prove rigorously the convergence of the above series. Fix $\theta\in (0,\eta_0)$, where $\eta_0$ is given in \eqref{eq:Hyper-u-v}. As $\psi^{(j)}$ is a solution to the Airy equation \eqref{eq:iter2-1}, we apply the estimate \eqref{est-L2-psi-energy} to obtain 
\begin{align}\label{psi1}
	c_i\|\psi^{(j)}\|_{L^2}+|\varepsilon|^{\frac12}c_i^{\frac12}\|\partial_Z\psi^{(j)}\|_{L^2}+|\varepsilon|c_i^{\frac12}\|(\partial_Z^2-\alpha^2)\psi^{(j)}\|_{L^2}&\leq C|\varepsilon|\|(\partial_Z^2-\alpha^2)\varphi^{(j-1)}\|_{L^2}\nonumber\\
	&\leq C|\varepsilon|\|(\partial_Z^2-\alpha^2)\varphi^{(j-1)}\|_{L^\infty_{\theta}},
\end{align}
by the virtue of embedding $L^\infty_{\theta}(\mathbb{R}_+)\hookrightarrow L^2(\mathbb{R}_+)$. Then using \eqref{psi1},  we obtain
$$\|\partial_Z(U'\psi^{(j)})\|_{L^\infty_{\theta}}\leq C\|\psi^{(j)}\|_{W^{1,\infty}}\leq C|\varepsilon|^{\frac14}c_i^{-\frac12}
\|(\partial_Z^2-\alpha^2)\varphi^{(j-1)}\|_{L^\infty_{\theta}}.$$
On the other hand, since $\varphi^{(j)}$ solves the Rayleigh equation \eqref{eq:iter2-2}, we apply the bound \eqref{eq-Ray-nonhom-tilde-Y} to $\varphi^{(j)}$ with $F=2\partial_Z(U'\psi^{(j)})$, and get
\begin{align}
	\|\varphi^{(j)}\|_{L^\infty_{\theta}}+\|\partial_Z\varphi^{(j)}\|_{L^\infty_{\theta}}&\leq C|\log c_i|\|\partial_Z(U'\psi^{(j)})\|_{L^\infty_{\theta}}\nonumber\\
	&\leq C|\varepsilon|^{\frac14}c_i^{-\frac12}|\log c_i|\|(\partial_Z^2-\alpha^2)\varphi^{(j-1)}\|_{L^\infty_{\theta}},\label{psi2}
\end{align}
and
\begin{align}	
	\|(\partial_Z^2-\alpha^2)\varphi^{(j)}\|_{L^\infty_{\theta}}&\leq \frac{C}{c_i}\|(U-c)(\partial_Z^2-\alpha^2)\varphi^{(j)}\|_{L^\infty_{\theta}}\nonumber\\
	&\leq C|\varepsilon|^{\frac14}c_i^{-\frac32}|\log c_i| \|(\partial_Z^2-\alpha^2)\varphi^{(j-1)}\|_{L^\infty_{\theta}}\label{psi3}.
\end{align}
Since $\alpha, c_i\sim O(1)$, then for sufficiently small $0<\nu\ll 1$, the factor
$$
 C|\varepsilon|^{\frac14} c_i^{-\frac32}|\log c_i|=C\nu^{\frac18}|\log c_i|c_i^{-\frac32}\alpha^{-\frac14}\ll 1.
$$
Thus, \eqref{psi3} implies
\begin{align}
	\sum_{j=0}^\infty \|(\partial_Z^2-\alpha^2)\varphi^{(j)}\|_{L^\infty_{\theta}}\leq C\|(\partial_Z^2-\alpha^2)\varphi^{(0)}\|_{L^\infty_{\theta}}\leq \frac{C|\log c_i|}{c_i}\|F\|_{L^\infty_{\theta}}.\label{psi4}
\end{align}
Substituting \eqref{psi4} into \eqref{psi1} and \eqref{psi2} respectively, we deduce that
$$
\begin{aligned}
	\sum_{j=1}^\infty\|\psi^{(j)}\|_{L^2}+\|\partial_Z\psi^{(j)}\|_{L^2}+\|(\partial_Z^2-\alpha^2)\psi^{(j)}\|_{L^2}&\leq Cc_i^{-\frac12}\sum_{j=0}^\infty\|(\partial_Z^2-\alpha^2)\varphi^{(j)}\|_{L^\infty_{\theta}}\\
	&\leq C|\log c_i| c_i^{-\frac32}\|F\|_{L^\infty_{\theta}},
\end{aligned}
$$
and
$$\begin{aligned}
	\sum_{j=0}^{\infty}	\|\varphi^{(j)}\|_{L^\infty_{\theta}}+\|\partial_Z\varphi^{(j)}\|_{L^\infty_{\theta}}&\leq C\sum_{j=0}^\infty\|(\partial_Z^2-\alpha^2)\varphi^{(j)}\|_{L^\infty_{\theta}}+C\|\varphi^{(0)}\|_{L^\infty_\theta}+C\|\partial_Z\varphi^{(0)}\|_{L^\infty_{\theta}}\nonumber\\
	&\leq \frac{C|\log c_i|}{c_i}\|F\|_{L^\infty_{\theta}}.
\end{aligned}
$$
Combining the above inequalities with \eqref{psi4} yields $H^2$-convergence of the series. Therefore, we arrive at the following result.
\begin{proposition}\label{prop:os-nonhom}
    Let $(\alpha,c)\in\mathbb{H}_2$. Then for any $F(Z)\in L^\infty_{\theta}$, there exists a solution $\phi\in H^2(\mathbb{R}_+)$ to \eqref{eq:non-OS} satisfying 
    \begin{align}
    \|\phi\|_{H^2} \leq Cc_i^{-\frac32}|\log c_i|\|F\|_{L^\infty_{\theta}}.\label{os}
    \end{align}
\end{proposition}

\subsection{Homogeneous Orr-Sommerfeld solution}
In this subsection, we construct the solution $\phi_s$ to 
a homogeneous Orr-Sommerfeld equation $\mathrm{OS}[\phi_s]=0.$ Recall the Rayleigh solution $\varphi_{Ray}$ obtained in Proposition \ref{prop:ray-homo}.

\begin{proposition}\label{prop-slowmode}
    Let $(\alpha,c)\in\mathbb{H}_2$. There exists a solution $\phi_s\in H^2(\mathbb{R}_+)$ to \eqref{eq:OSeq} satisfying
    \begin{align}
    \|\phi_s-\varphi_{Ray}\|_{H^2}\leq C\nu^{\frac12-}.\label{bd1}
    \end{align}
\end{proposition}
\begin{proof}
    We seek a solution of the form $\phi_s=\varphi_{Ray}+\tilde{\phi}_s$, where the remainder $\tilde{\phi}_s$ satisfies 
    \begin{align}
        OS[\tilde{\phi}_s]=\varepsilon(\partial_Z^2-\alpha^2)^2\varphi_{Ray}.\label{os1}
    \end{align}
Note that 
    \begin{align*}
        (\partial_Z^2-\alpha^2)^2\varphi_{Ray}=&(\partial_Z^2-\alpha^2)\left(\frac{U''\varphi_{Ray}}{U-c}\right)\\ =&\partial_Z^2\left(\frac{U''}{U-c}\right)\varphi_{Ray}+2\partial_Z\left(\frac{U''}{U-c}\right)\partial_Z\varphi_{Ray}\\ &~+\frac{ U''}{U-c}(\partial_Z^2-\alpha^2)\varphi_{Ray}. 
    \end{align*}
Then it follows that
$$
\begin{aligned}
\|(\partial_Z^2-\alpha^2)^2\varphi_{Ray}\|_{L^\infty_{\eta_0}}
&\leq Cc_i^{-3}\|\varphi_{Ray}\|_{L^\infty}+Cc_i^{-2}\|\partial_Z\varphi_{Ray}\|_{L^\infty}+Cc_i^{-1}\|(\partial^2_Z-\alpha^2)\varphi_{Ray}\|_{L^\infty}\nonumber\\
&\leq Cc_i^{-3},
\end{aligned}
$$
where we have used the bound \eqref{bd} in last inequality. Applying Proposition \ref{prop:os-nonhom} to \eqref{os1} with $F=\varepsilon (\partial_Z^2-\alpha^2)^2\varphi_{Ray}$, we obtain, for sufficiently small $0<\nu\ll 1$, that
\begin{align}
	\|\tilde{\phi}_{s}\|_{H^2}\leq C|\varepsilon| c_i^{-\frac92}|\log c_i|\leq C\nu^{\frac12-}.\nonumber
\end{align}
The proof of Proposition \ref{prop-slowmode} is complete.
\end{proof}

 \subsection{Viscous sublayer}
In order to correct $\partial_Z\phi_s(0)$, we construct a solution $\phi_{f}$ to the homogeneous Orr-Sommerfeld equation \eqref{eq:non-OS} that exhibits a boundary layer structure near $Z=0.$
\begin{proposition}\label{osf}
	Let $(\alpha,c)\in \mathbb{H}_2$. There exists a solution $\phi_f(Z)\in H^2(\mathbb{R}_+)$ to \eqref{eq:non-OS}, such that
	\begin{equation}
		\phi_f(0)=1,\qquad \partial_Z\phi_f(0)=-\frac{|c|^{\frac12}}{|\varepsilon|^{\frac12}}\left(e^{-\frac{1}{4}\pi \mathrm{i}}+o(1)\right).\label{fast}
\end{equation}
Moreover, $\phi_f(Z)$ satisfies the following bound:
 \begin{align}
 \frac{|c|^{\frac14}}{|\varepsilon|^{\frac14}}\|\phi_{f}\|_{L^2}+\|\phi_f\|_{L^\infty}+\frac{|\varepsilon|^{\frac14}}{|c|^{\frac14}}\|\partial_Z\phi_{f}\|_{L^2}+\frac{|\varepsilon|^{\frac34}}{|c|^{\frac34}}\|\partial_Z^2\phi_{f}\|_{L^2}\leq C.\label{bf}
 \end{align}
\end{proposition}
\begin{proof}

 Following the idea from \cite{GMM-duke}, we construct $\phi_f$ around an exponential profile:  $$\phi_{f,0}(Z)\triangleq e^{-\omega Z},\quad\omega=\left(\frac{-c}{\varepsilon}\right)^{\frac12}=\left(\frac{c_i-\mathrm{i}c_r}{|\varepsilon|}\right)^{\frac12},$$
which is a solution to $$-\varepsilon\partial_Z^4\phi_{f,0}-c\partial_Z^2\phi_{f,0}=0.$$
Since $c_i\ll c_r$, one has
\begin{align}
	\omega=\frac{|c|^{\frac12}}{|\varepsilon|^{\frac12}}\left(e^{-\frac{1}{4}\pi \mathrm{i}}+o(1)\right).\nonumber
\end{align}
Let $\kappa\triangleq\frac{1}{|\omega|}=\frac{|\varepsilon|^{\frac12}}{|c|^{\frac12}}$ be the scale of sublayer. In the rescaled variable $\xi\triangleq \frac{Z}{\kappa}$, the Orr-Sommerfeld equation \eqref{eq:OSeq} can be rewritten in the following way:
\begin{align}
	-\partial_\xi^4\phi+\frac{\omega^2}{|\omega|^2}\partial_{\xi}^2\phi&=
	-2\alpha^2\kappa^2\partial_\xi^2\phi-\varepsilon^{-1}\kappa^{2}U\partial_\xi^2\phi+\alpha^4\kappa^4\phi+\varepsilon^{-1}\alpha^2\kappa^4(U-c)\phi+\varepsilon^{-1}\kappa^4U''\phi\nonumber\\
	&=\partial_\xi^2g(\phi),\nonumber
\end{align}
where
$$
\begin{aligned}
g(\phi)=&-2\alpha^2 \kappa^2\phi-\varepsilon^{-1}\kappa^2U(\kappa\xi)\phi-2\varepsilon^{-1}\kappa^3\int_\xi^\infty U'(\kappa\xi')\phi(\xi')d \xi'\nonumber\\
&+\varepsilon^{-1}\alpha^2\kappa^4\int_\xi^\infty\int_{\xi'}^\infty (U(\kappa\xi'')-c+\varepsilon \alpha^2)\phi(\xi'')d \xi''d\xi'.
\end{aligned}
$$
Then we look for the solution $\tilde{\phi}_{f}$ in the following form:
\begin{align}
	\tilde{\phi}_f(Z)=\phi_{f,0}+\sum_{j=1}^N\phi_{f,j}(\kappa^{-1}Z)+\phi_{f,R}(Z), \label{phit}
\end{align}
where $\phi_{f,j}$ are solved from the following hierarchy:
\begin{align}
	-\partial_\xi^2\phi_{f,j}+\frac{\omega^2}{|\omega|^2}\phi_{f,j}=g(\phi_{f,j-1}),~j\geq 1.\nonumber
\end{align}
Inductively, we obtain
\begin{align}
	\phi_{f,j}(\xi)=\frac{|\omega|}{2\omega}\int_0^\xi e^{-\frac{\omega}{|\omega|}(\xi-\xi')}g(\phi_{f,j-1})(\xi')d\xi'+\frac{|\omega|}{2\omega}\int_{\xi}^\infty e^{-\frac{\omega}{|\omega|}(\xi'-\xi)}g(\phi_{f,j-1})(\xi')d\xi'.\label{phij}
\end{align}
 Then for any $\theta_0\in \left(0, \text{Re}\frac{\omega}{|\omega|}\right)$, we deduce  from \eqref{phij} that
\begin{align}
	\sup_{\xi}|e^{\theta_0\xi}\phi_{f,j}(\xi)|\leq C\sup_{\xi}|e^{\theta_0\xi} g(\phi_{f,j-1})(\xi)|.\label{phij1}
\end{align} 
Note that
\begin{align}
	\sup_{\xi}|e^{\theta_0\xi}g(\phi)(\xi)|&\leq C|\varepsilon|^{-1}\kappa^3\sup_\xi|(1+\xi)e^{\theta_0\xi}\phi(\xi)|\nonumber\\
	&\leq C|\varepsilon|^{\frac12}|c|^{-\frac32}\sup_\xi|(1+\xi)e^{\theta_0\xi}\phi(\xi)|,\nonumber
\end{align}
where we have used $|U(\kappa\xi)|\leq C\kappa\xi\|U'\|_{L^\infty}\leq C\kappa\xi$ in the first inequality. Substituting it into \eqref{phij1} yields
\begin{align}\label{phij2}
	\sup_{\xi}|e^{\theta_0\xi}\phi_{f,j}(\xi)|\leq C(|\varepsilon|^{\frac12}|c|^{-\frac32})^{j}	\sup_{\xi}|(1+\xi)^je^{\theta_0\xi}\phi_{f,0}(\xi)|\leq C(|\varepsilon|^{\frac12}|c|^{-\frac32})^{j}.
\end{align}
Similarly, it holds that for $k=1,2$, 
\begin{align}
	\sup_{\xi}|e^{\theta_0\xi}\partial_\xi^k\phi_{f,j}(\xi)|\leq C(|\varepsilon|^{\frac12}|c|^{-\frac32})^{j}.\label{phij3}
\end{align}

The remainder $\phi_{f,R}$ solves the Orr-Sommerfeld equation
\begin{align}
	\mathrm{OS}[\phi_{f,R}]=&2\varepsilon \kappa^{-2}\alpha^2\partial_\xi^2\phi_{f,N}+\kappa^{-2}U\partial_\xi^2\phi_{f,N}-\alpha^2(U-c)\phi_{f,N}-\varepsilon\alpha^4\phi_{f,N}-U''\phi_{f,N}\nonumber\\
	:=& E_{N}.\nonumber
\end{align}
Using the bounds \eqref{phij2} and \eqref{phij3} for $\phi_{f,N}$, we obtain
$$\begin{aligned}
\|E_N\|_{L^\infty_{\theta}}&\leq C\left(|\varepsilon|\kappa^{-2}\alpha^2 +\kappa^{-2}\right)\|\partial_\xi^2\phi_{f,N}\|_{L^\infty_{\theta}}
+\left(\alpha^2+|\varepsilon|\alpha^4+1\right)\|\phi_{f,N}\|_{L^\infty_{\theta}}\\
&\leq C\kappa^{-2}(|\varepsilon|^{\frac{1}{2}}|c|^{-\frac32})^N\leq C\nu,
\end{aligned}
$$
provided that $N\geq 7.$ Then applying \eqref{os} to $\phi_{f,R}$, we obtain
\begin{align}
	\|\phi_{f,R}\|_{H^2}\leq C\nu c_i^{-\frac32}|\log c_i|\leq C\nu^{\frac34},\label{phij4}
\end{align}
by the virtue of $0<\nu\ll1$. 

We now define the boundary layer solution by normalizing $\tilde{\phi}_f$ as follows:
 $$\phi_{f}(Z)=\frac{\tilde{\phi}_f(Z)}{\tilde{\phi}_f(0)}.$$
From the decomposition \eqref{phit}, the bounds \eqref{phij2} of $\phi_{f,j}$, and the estimate \eqref{phij4} for the remainder $\phi_{f,R}$, we obtain
\begin{align}\label{s1}
	\tilde{\phi}_{f}(0)=1+O(|\varepsilon|^{\frac12}|c|^{-\frac32}),
\end{align}
and
\begin{align}\label{s2}
	\partial_Z\tilde{\phi}_{f}(0)=-\omega+O(|c|)=\frac{|c|^{\frac12}}{|\varepsilon|^{\frac12}}\left(e^{-\frac{1}{4}\pi \mathrm{i}}+o(1)\right).
\end{align}
By combining \eqref{s1} with \eqref{s2}, we obtain \eqref{fast}. The estimate \eqref{bf} is a direct consequence of \eqref{phij3} and \eqref{phij4}. This concludes the proof of Proposition \ref{osf}.
\end{proof}

\subsection{Viscous instability}
In this subsection, we demonstrate that the instability established in Theorem \ref{iv} persists for the Orr-Sommerfeld equation with small viscosity $0<\nu\ll1$.

\begin{theorem}\label{thm-OSeq-solution}
    Let $\alpha\ll1$. There exists $\nu_0\in (0,1)$ such that for any $\nu\in (0,\nu_0)$, there exists $c_{os}\in\mathbb{C}$ for which the homogeneous Orr-Sommerfeld equation
    \begin{equation}\label{oseq}
    	\left\{
    	\begin{aligned}
    	&-\varepsilon (\partial_Z^2-\alpha^2)^2\phi +(U-c_{os})(\partial_Z^2-\alpha^2)\phi-U''\phi=0,\\
    	&\phi(0)=\partial_Z\phi(0)=0,
    	\end{aligned}
    \right.
    \end{equation}
admits a non-trivial solution $\phi\in H^2(\mathbb{R}_+)$ satisfying
\begin{align}
	\|\phi\|_{L^\infty}+\alpha^{\frac12}\|\phi\|_{L^2}+\|\partial_Z\phi\|_{L^2}+\frac{|\varepsilon|^{\frac14}}{|c|^{\frac14}}\|\partial_Z^2\phi\|_{L^2}\leq C.\label{bdphi}
\end{align}
Moreover, $c_{os}$ satisfies the following estimate
\begin{align}
	|c_{os}-c_{Ray}|\leq C\nu^{\frac14-},\nonumber
\end{align}
where $c_{Ray}$ is the inviscid eigenvalue constructed in Theorem \ref{iv}.
\end{theorem}
\begin{proof}
	Recall the slow mode $\phi_{s}$ and the viscous boundary layer $\phi_{f}$ constructed in Propositions \ref{prop-slowmode} and \ref{osf} respectively.	We look for the solution $\phi$ of the form
	\begin{align}
		\phi(Z)=\phi_{s}(Z)-\frac{\partial_Z\phi_{s}(0)}{\partial_Z\phi_f(0)}\phi_{f}(Z).\nonumber
	\end{align}
Thus, $\phi(Z)$ satisfies the full boundary conditions in the Orr-Sommerfeld equation \eqref{oseq} if and only if $c$ is a solution to the following dispersion relation
\begin{align}
	\phi(0;c)=\phi_s(0;c)-\frac{\partial_Z\phi_{s}(0;c)}{\partial_Z\phi_f(0;c)}\phi_{f}(0;c)=0.\label{phi0}
\end{align}

	Let $D_0=\{c\in \mathbb{C}\mid |c-c_{Ray}|\leq \nu^{\delta}\}$ with $\delta\in (0,\frac14)$. Note that $\phi_{Ray}(0;c_{Ray})=0$. On the boundary $\partial D_{0}$, one has	
	\begin{align}
		|\varphi_{Ray}(0;c)|&=|\phi_{Ray}(0;c)-\varphi_{Ray}(0;c_{Ray})|\nonumber\\
		&=\left(\frac{1}{U_\infty^2}+O(1)(\alpha|\log c_i|+|c|)\right)|c-c_{Ray}|\nonumber\\
		&\geq \frac{\nu^{\delta}}{2U_{\infty}^2},\nonumber
	\end{align}
where the bound \eqref{phicd} for $\partial_c\varphi_{Ray}(0;c)$ has been used in the second inequality. From the bounds \eqref{R3}, \eqref{R4} \eqref{bd1} and \eqref{fast}, we obtain
$$
\begin{aligned}
|\phi(0;c)-\varphi_{Ray}(0;c)|&\leq \|\phi_s-\varphi_{Ray}\|_{H^2}+\frac{|\partial_Z\phi_s(0;c)|}{|\partial_Z\phi_{f}(0;c)|}\\
&\leq C\left(\nu^{\frac12-}+|\varepsilon|^{\frac12}|c|^{-\frac12}\right)\nonumber\\
&\leq C\nu^{\frac14}|\log\nu|\leq \frac12|\varphi_{Ray}(0;c)|,
\end{aligned}
$$
provided that $0<\nu\ll 1$. Moreover, the analyticity of $\phi(0;c)$ in $c$ follows from $L^\infty$-convergence of Rayleigh-Airy iteration. Therefore, by Rouch\'e Theorem, we can find a zero point $c_{os}\in \mathbb{D}_0$ of the equation \eqref{phi0}. The bound \eqref{bdphi} follows from  \eqref{bd}, \eqref{bd1}, and \eqref{bf}. The proof of Theorem \ref{thm-OSeq-solution} is complete.
\end{proof}
\section{Recovery of tangential velocities}\label{sec-recover}
As mentioned in the roadmap for the construction of $(\tilde{u},\tilde{v},\tilde{w})$, the vorticity in $XY-$plane: $\Psi(Z)=i\beta\tilde{u}(Z)-i\sigma\tilde{v}(Z)$ can be obtained by solving the Airy equation \eqref{eq:Airy-vorticity}. The tangential velocity components $\tilde{u}$ and $\tilde{v}$ can then be recovered via the relations
\begin{align}\label{eq:u-v-relation-w-psi}
\left\{
	\begin{aligned}
		&\mathrm{i}\sigma\tilde{u}(Z)+\mathrm{i}\beta\tilde{v}(Z)=-\partial_Z \tilde{w}(Z),\\
		&\mathrm{i}\beta\tilde{u}(Z)-\mathrm{i}\sigma\tilde{v}(Z)=\Psi(Z),
	\end{aligned}
	\right.
\end{align}
where $\sigma=\alpha\lambda_1,~\beta=\alpha\lambda_2$, with $\lambda_1$ and $ \lambda_2$ are given in Proposition \ref{pSC}. \\

{\bf Proof of Theorem \ref{main-result}:} In Theorem \ref{thm-OSeq-solution}, we find the unstable eigenvalue $c_{os}$, and construct the solution $\tilde{w}$ to the homogeneous Orr-Sommerfeld equation, which corresponds to the vetical velocity field $\tilde{w}(Z)$. From \eqref{bdphi}, we obtain
    \begin{align}\label{wb}    
       	\|\tilde{w}\|_{L^\infty}+\alpha^{\frac12}\|\tilde{w}\|_{L^2}+\|\partial_Z\tilde{w}\|_{L^2}+\nu^{\frac18}\alpha^{-\frac12}\|\partial_Z^2\tilde{w}\|_{L^2}\leq C.
    \end{align}
It suffices to establish $H^1$-bounds for $(\tilde{u},\tilde{v})$. Since $\Theta(Z)$ solves the Airy equation \eqref{eq:Airy-vorticity}, and 
   \begin{align*}
        &\left\|(-\lambda_1\partial_Zv_s+\lambda_2\partial_Z u_s)\tilde{w}\right\|_{L^2}\leq C\|\tilde{w}\|_{L^\infty}\leq C.
        \end{align*}
  Then using the Airy bound \eqref{est-L2-psi-energy} to \eqref{eq:Airy-vorticity}, we obtain
  \begin{equation}
    \begin{aligned}\label{Theta}
      \|\Psi\|_{L^2}&\leq \frac{C}{c_i}\left\|(-\lambda_1\partial_Zv_s+\lambda_2\partial_Z u_s)\tilde{w}\right\|_{L^2}\leq \frac{C}{c_i}\leq \frac{C}{\alpha^2},\\
      \|\partial_Z\Psi\|_{L^2}&\leq \frac{C}{|\varepsilon|^{\frac12}c_i^{\frac12}}\left\|(-\lambda_1\partial_Zv_s+\lambda_2\partial_Z u_s)\tilde{w}\right\|_{L^2}\leq \frac{C}{|\varepsilon|^{\frac12}c_i^{\frac12}}\leq \frac{C}{\nu^{\frac14}\alpha^{\frac12}}.
    \end{aligned}
    \end{equation}
   Then from \eqref{eq:u-v-relation-w-psi}, we recover the tangential velocity components $(\tilde{u}, \tilde{v})$ as follows.
    \begin{align*}
        \tilde{u}(Z)=\frac{-\sigma \partial_Z\tilde{w}+\beta\Psi}{\mathrm{i}(\sigma^2+\beta^2)} =\frac{-\lambda_1\partial_Z\tilde{w}+\lambda_2\Psi}{\mathrm{i}\alpha},\quad \tilde{v}(Z)=\frac{-\beta \partial_Z\tilde{w}-\sigma\Psi}{\mathrm{i}(\sigma^2+\beta^2)}
        =\frac{-\lambda_2\partial_Z\tilde{w}-\lambda_1\Psi}{\mathrm{i}\alpha}.
    \end{align*}
  Therefore, from \eqref{wb} and \eqref{Theta}, we obtain
  $$
  \begin{aligned}
  	\|(\tilde{u},\tilde{v},\tilde{w})\|_{L^2}&\leq C\alpha^{-3},\\
  	\|(\partial_Z\tilde{u},\partial_Z\tilde{v},\partial_Z\tilde{w})\|_{L^2}&\leq C\nu^{-\frac14}\alpha^{-\frac32}.
  \end{aligned}
  $$
 Thus, the proof of Theorem \ref{main-result} is complete.
   \qed

\section*{Acknowledgments}
This work is supported by the National Key R\&D Program of China under the grant 2023YFA1010300. C.-J. Liu is supported by the National
Natural Science Foundation of China No. 12271358, 12331008 and 12250710674. 
M. Ma is partially supported by NSF of China under Grant 12371230. D. Wu is supported by NSF of China under Grant 12471196.
Z. Zhang is supported by the Hong Kong Research Grant Council ECS grant (Project No. 25303523) and HK PolyU internal funding P0045335.


\end{document}